\documentclass[11pt]{amsart}

\usepackage{amssymb,amsmath,enumerate,epsfig}
\usepackage{amsfonts,color}
\usepackage{a4,latexsym}
\usepackage{parskip}

\usepackage{hyperref}
\hypersetup{pdftitle=Jacobian genus 1 fibrations}

\newcommand{\C} {\mathbb{C}}
\newcommand{\Q} {\mathbb{Q}}
\newcommand{\N}  {\mathbb{N}}

\newcommand{\F}{\mathbb{F}}
\newcommand{\Z}{\mathbb{Z}}

\newcommand{\OO}{\mathcal{O}}

\newcommand{\PP}{\mathbb{P}}
\newcommand{\NS}{\mathop{\rm NS}\nolimits}
\newcommand{\NeS}{N\'eron-Severi}
\newcommand{\MW}{\mbox{MW}}

\newcommand\MoW{Mordell-\kern-.2exWeil}

\newcommand{\A}{\mathbb{A}}

\newcommand{\ta}{\tilde A}
\newcommand{\td}{\tilde D}
\newcommand{\te}{\tilde E}
\newcommand{\0}{^{\phantom0}}

\newtheorem{Theorem}{Theorem}
\newtheorem{Proposition}[Theorem]{Proposition}

\theoremstyle{remark}
\newtheorem{Remark}[Theorem]{Remark}

\newtheorem{Example}[Theorem]{Example}

\theoremstyle{definition}

\begin{document}

\title[Genus 1 fibrations]{Genus 1 fibrations
on the supersingular K3 surface in characteristic 2 with Artin invariant 1}

\author{Noam D.~Elkies}
\address{
Mathematics Department,
Harvard University,
1 Oxford Street,
Cambridge, MA 02138,
USA}
\email{elkies@math.harvard.edu}
\urladdr{http://www.math.harvard.edu/$\sim$elkies}

\author{Matthias Sch\"utt}
\address{Institut f\"{u}r Algebraische Geometrie, Leibniz Universit\"{a}t
  Hannover, Welfengarten 1, 30167 Hannover, Germany}
\email{schuett@math.uni-hannover.de}
\urladdr{http://www.iag.uni-hannover.de/$\sim$schuett/}



\thanks{Partial support from NSF under grant DMS-0501029
and ERC under StG~279723 (SURFARI) gratefully acknowledged.
}

\date{March 20, 2014}

 \begin{abstract}
The supersingular K3 surface~$X$\/ in characteristic~$2$
with Artin invariant~$1$ admits several genus~$1$ fibrations
(elliptic and quasi-elliptic).  We use a bijection between fibrations and
definite even lattices of rank~$20$ and discriminant~$4$
to classify the fibrations, and we exhibit isomorphisms between the
resulting models of~$X$.
We also study a configuration of $(-2)$-curves on~$X$\/ related 
to the incidence graph of
points and lines of $\PP^2(\F_4)$.

 \end{abstract}

\maketitle

\section{Introduction}

Elliptic fibrations are a versatile tool for studying algebraic surfaces.
One of their key advantages is that one can often compute the \NeS\ lattice,
and in particular the Picard number, in a systematic way.
This has been carried out with great success in the study of K3 surfaces.
There is one feature that singles out K3 surfaces among all algebraic
surfaces admitting elliptic fibrations:
a single K3 surface may admit several distinct elliptic fibrations.


Several previous papers classify all jacobian elliptic fibrations 
on a given class of K3 surfaces (i.e.~elliptic fibrations with section).
Oguiso determined all jacobian elliptic fibrations of a Kummer
surface of two non-isogenous elliptic curves~\cite{Oguiso}.
This classification was achieved by geometric means.
Subsequently Nishiyama proved Oguiso's result again
by a lattice theoretic technique~\cite{Nishi}.
Equations and elliptic parameters were derived by Kuwata and Shioda \cite{KS}.
Nishiyama also considered other Kummer surfaces of product type and certain
singular K3 surfaces.  Kumar recently determined all elliptic fibrations on
the Kummer surface of the Jacobian of a generic curve of genus~$2$~\cite{Kumar}.

All these classifications are \emph{a priori} only valid in characteristic zero.
In this paper we present a classification that is specific
to positive characteristic and does not miss any non-jacobian fibrations.
Namely we consider the supersingular K3 surface $X$ in characteristic
2 with Artin invariant 1. 
In this setting we must deal with quasi-elliptic fibrations
whose general fiber is a cuspidal rational curve.
As a uniform notation, we shall refer to either an elliptic or a
quasi-elliptic fibration as a \emph{genus~1 fibration}.

\begin{Theorem}
\label{thm}
Let $X$ denote the supersingular K3 surface $X$ with Artin invariant~1
over an algebraically closed field of characteristic~2.
Then $X$ admits exactly 18  genus~1 fibrations.
\end{Theorem}

A crucial ingredient of our main result is Theorem \ref{thm:jac}
stating that any genus 1 fibration on $X$ admits a section.
The classification of all possible fibrations is then achieved
in Section \ref{s:class} by lattice theoretic means
\emph{\`{a} la} Kneser-Nishiyama (cf.~Section \ref{s:Nishi}).
We also determine whether the fibrations are elliptic or quasi-elliptic
using a criterion developed in Section \ref{s:e-qe} (Theorem \ref{thm:e-qe}). 
The existence of these fibrations on~$X$ is established by exhibiting
an explicit Weierstrass form over the prime field $\F_2$ for each of them.
We shall furthermore connect all fibrations by explicit isomorphisms
over $\F_4$ (usually even over $\F_2$, but we shall see that
this is not always possible).
Equations and isomorphisms are given in Section \ref{s:eq}.
The uniqueness part of Theorem \ref{thm} is proven in Section \ref{s:unique}
by working with explicit Weierstrass equations.
Section \ref{s:config} shows that some specific $(-2)$ curves on $X$
generate the incidence graph of points and lines in $\PP^2(\F_4)$.
We derive some surprising consequences for configurations in  $\PP^2(\F_4)$
such as the absence of a $14$-cycle.
The paper concludes with comments on the implications of our classification
for reduction from characteristic zero.

\section{The supersingular K3 surface in characteristic  2 with
Artin invariant 1}
\label{s:ss}

On an algebraic surface $S$,
we consider the \NeS\ group $\NS(S)$ consisting of divisors
up to algebraic equivalence.
The \NeS\ group is finitely generated and abelian; its
rank  is called the Picard number of $S$ and denoted by $\rho(S)$.
The intersection form endows $\NS(S)$ with the structure of an
integral lattice up to torsion.
By the Hodge index theorem, this lattice has signature $(1,\rho(S)-1)$.
On a K3 surface, algebraic and numerical equivalence are the same.
Hence $\NS(S)$ is torsion-free and thus a lattice in the strict
sense.

In characteristic zero,
Lefschetz' theorem bounds the Picard number by the central Hodge
number:
\begin{eqnarray}
\label{eq:Lef}
\rho(S) \leq h^{1,1}(S).
\end{eqnarray}
In positive characteristic, however, we have only
Igusa's theorem which gives the weaker upper bound:
\begin{eqnarray}
\label{eq:Igusa}
\rho(S)\leq b_2(S).
\end{eqnarray}
Surfaces attaining equality in the former bound (\ref{eq:Lef})
are sometimes called singular (in the sense of exceptional,
as with elliptic curves said to be ``singular'' when they have
complex multiplication).  Equality in the latter bound~(\ref{eq:Igusa})
leads to Shioda's notion of supersingular surfaces.

For K3 surfaces, one has $h^{1,1}(S)=20$ and $b_2(S)=22$.
Supersingular K3 surfaces were studied by Artin in \cite{Artin}.
In particular he proved that for a supersingular K3 surface in
characteristic $p$, the \NeS\ group $\NS(S)$ has discriminant
\begin{eqnarray}
\label{eq:disc}
\mbox{disc}(\NS(S))=-p^{2\sigma}, \;\;\; 1\leq \sigma\leq 10.
\end{eqnarray}
Here $\sigma$ is usually called the Artin invariant of $S$.
Artin also derived a stratification of the moduli space of supersingular
K3 surfaces in terms of~$\sigma$.
This classification was later complemented by Ogus
who proved that there is a unique supersingular K3 surface with
$\sigma=1$ over the algebraic closure of the base field \cite{Ogus}
(see \cite{RS2} for characteristic $2$).

From here on we specialize to characteristic $p=2$.
There are several known models for the unique supersingular K3
surface $X$ with $\sigma=1$ (e.g.~\cite{DK}, \cite{KK}, \cite{RS2}, \cite{S-MJM}).
For instance one can take the following~genus one fibration
from \cite{DK} with affine parameter $t\in\PP^1$:
\[
X:
y^2 = x^3 + t^3x^2+t.
\]
This fibration is quasi-elliptic, i.e.~all fibers are singular curves (see Section \ref{s:g=1}),
but it has only one reducible fiber.
The special fiber is located outside the affine chart on the base curve $\PP^1$,
at $t = \infty$, and has Kodaira type $I_{16}^*$.
It follows that there can be no sections other than the zero section
$O$, and that
\[
\NS(X) = U \oplus D_{20}.
\]
This fibration will reappear in our classification in Sections
\ref{s:class}--\ref{s:eq} as \#18.
Note that a singular fiber of type $I_{16}^*$ is impossible for
a jacobian genus~$1$ fibration on any K3 surface outside characteristic
two, for otherwise
the surface would contradict either \eqref{eq:Lef} or \eqref{eq:disc}.
In comparison, for an elliptic K3 surface in characteristic two,
the maximal singular fiber types are $I_{13}^*$ and $I_{18}$ by
\cite{S-max}.


\section{Genus one fibrations}
\label{s:g=1}

A genus~$1$ fibration on a smooth projective surface $S$ is a surjective
morphism onto a smooth curve $C$\/ such that the general fiber $F$
is a curve of arithmetic genus~$1$.
If the characteristic is different from $2$ and $3$, then this
already implies that $F$ smooth.
In the presence of a section, $F$\/ is an elliptic curve;
hence these fibrations are called elliptic.
In characteristics $2$ and $3$, however, $F$\/ need not be smooth,
it may be a cuspidal rational curve.
Such a fibration is called quasi-elliptic.

For general properties of genus~$1$ fibrations (mostly elliptic),
the reader is referred to the recent survey \cite{SSh} and the
references therein, specifically \cite{CD}.
We shall review a few more details about quasi-elliptic fibrations
in Section \ref{s:unique}.
Here we only recall two useful formulas.
The first computes the Euler-Poincar\'e characteristic $e(S)$ through
the (reducible) singular fibers.
The sum includes a local correction term that accounts for
the wild ramification $\delta_v$ in the case of an elliptic surface,
and for the non-zero Euler-Poincar\'e characteristic of the general fiber
in the case of a quasi-elliptic surface:
\begin{itemize}
\item
$S$ elliptic:\phantom{-quasi}\;\;\,
$
 e(S) = \sum_{v\in C} (e(F_v) + \delta_v)
 $,
 \item
 $S$ quasi-elliptic:\;\;
$
 e(S) = e(C)e(F) + \sum_{v\in C} (e(F_v) - 2).
$
\end{itemize}
The Shioda-Tate formula concerns jacobian genus~$1$ fibrations.
It asserts that the \NeS\ group is generated by fiber components and sections.
Outside the \MoW\ group, the only relation is that any two
fibers are algebraically equivalent.

In order to find a genus~$1$ fibration on a K3 surface, it suffices
to find a divisor $D$\/ of zero self-intersection $D^2=0$ by \cite{PSS}.
Then either $D$\/ or $-D$\/ is effective by Riemann-Roch, and the linear system
$|D|$ or $|-D|$ induces a genus~$1$ fibration (usually elliptic).
If the divisor $D$ has the shape of a singular fiber from Kodaira's list,
then it in fact appears as a singular fiber of the given fibration.
Moreover, any irreducible curve $C$ with $C\cdot D=1$ gives a section
of the fibration.

In the K3 case, any curve has even self-intersection by the adjunction
formula, so $C^2$ is even.  Hence $C$ and $D$ span the hyperbolic plane $U$.
In summary, a jacobian elliptic fibration on a K3 surface is realized
by identifying a copy of $U$\/ inside $\NS$.  (Warning: in general it might
not be the copy of~$U$\/ we started with, because the sections of~$D$\/ may
have a base locus.  But it is always the image of the original copy of~$U$\/
under an isometry of $\NS(S)$.)
We now prove a result which implies that any genus one fibration on $X$ is jacobian:

\begin{Theorem}
\label{thm:jac}
Any genus 1 fibration on a supersingular K3 surface of Artin invariant 1 admits a section.
\end{Theorem}

\begin{proof}
Let $X$ denote the supersingular K3 surface of Artin invariant 1 in characteristic $p$.
Given a genus 1 fibration, we denote the class of a fiber by $F$ 
and the multisection index by $m\in\N$.
That is, 
\[
m\Z = \{D.F, \;\; D\in\NS(X)\}.
\]
Then the fibration has a section if and only if $m=1$. Assume $m>1$.
Then $F/m\in\NS(X)^\vee$, and in fact
\[
N:= \langle \NS(X), F/m\rangle \;\; \text{is an even integral lattice,}
\]
since $F^2=0$.
Presently $F$ is indivisible in $\NS(X)$ since there cannot be any multiple fibers 
by the canonical bundle formula (see \cite[Thm.~6.8]{SSh}).
Hence $\NS(X)$ has index $m$ in $N$ from which we infer
\[
\mbox{disc} (N) = \mbox{disc}(\NS(X))/m^2.
\]
Since the discriminant is an integer,
it follows at once that $m=p$.
But even then, $N$ is a unimodular lattice of signature $(1,21)$ 
which gives a contradiction.
\end{proof}

\begin{Remark}
The above argument may be applied to any elliptic surface with indivisible fiber class.
In fact, one may compare Keum's result for complex elliptic K3 surfaces \cite{Keum}
which states in the analogous notation that $\NS(\mbox{Jac}(X))=N$.
\end{Remark}


Throughout this paper we shall employ the following terminology.
Kodaira's notation for singular fibers of type $I_n$ (and $III, IV$) will be used
interchangeably with the corresponding extended Dynkin diagrams
$\tilde A_{n-1}$ or the root lattices $A_{n-1}$,
and likewise for $\td_n, D_n (n\geq 4)$ and $\te_n, E_n (n=6,7,8)$.
In principle, there is an ambiguity for $A_1$ and $A_2$,
but throughout this paper the root lattice will in fact
determine the fiber type uniquely.
The zero section will be denoted by $O$.
The fiber component meeting $O$ is called the identity component.
For other simple components, we use the self-explanatory termini
far component ($\td_n (n>4), \te_6, \te_7$), near component ($\td_n
(n>4)$) and opposite component as well as
even and~odd components ($\ta_n, n$ odd).

\section{Elliptic vs.~quasi-elliptic fibrations}
\label{s:e-qe}

We have already mentioned the subtlety in characteristics $p=2$ and $3$
that there are quasi-elliptic fibrations.
This brings us to the question how to detect from $\NS=U+M$
whether the corresponding  genus~$1$ fibration is elliptic
or quasi-elliptic.  In this section, we shall discuss a few criteria.

A first criterion comes from the singular fibers:
namely a quasi-elliptic fibration does not admit multiplicative fibers.
The additive fiber types are also restricted:
\begin{itemize}
\item
no
$IV, IV^*, I_n^* ~(n>0$ odd$)$ in characteristic $2$,
\item
no $III, III^*$ or $I_n^* ~(n\geq 0)$ in characteristic $3$.
\end{itemize}
The Euler-Poincar\'e characteristic gives a second simple approach
to distinguish elliptic and  quasi-elliptic fibration:
on a quasi-elliptic fibration, only the reducible singular fibers
contribute to $e(X)$ (which can also be computed as alternating
sum of Betti numbers or with Noether's formula).
If the sum over the fibers indeed returns the right number, then
we can compare to the sum without the correction terms for the
general fiber (plus possibly wild ramification which necessarily
is non-zero for certain fiber types by \cite{SS2}).
If the latter sum exceeds $e(X)$, then the fibration cannot be
elliptic.
This criterion can be very useful because the reducible singular
fibers are visible in $\NS(X)$ by the Shioda-Tate formula.

The perhaps  most general approach relies on the fact that quasi-elliptic
surfaces are always unirational, hence supersingular.
On the other hand, the $\MW$-group of a quasi-elliptic fibration
is always finite and in fact $p$-elementary (i.e.~isomorphic to $(\Z/p\Z)^r$ for some $r\in\N$).
This leads to the following criterion:

\begin{Theorem}[Rudakov-Shafarevich {\cite[\S4]{RS}}]
\label{Thm:PSS}
Given a genus~1 fibration on some algebraic surface $X$
with $\chi(\mathcal O_X)>1$ in characteristic
$p$, not necessarily jacobian.
This fibration is
quasi-elliptic if and only if the following conditions are satisfied:
\begin{enumerate}[(i)]
\item
$p=2,3$,
\item
the root lattice of each reducible fiber has $p$-elementary discriminant group,
\item
the fiber components generate a sublattice of  $\NS(X)$ of corank
one.
\end{enumerate}
\end{Theorem}

Specifically this implies for a jacobian quasi-elliptic fibration
that the \MoW\ group is $p$-elementary because the fibers
do not accommodate any higher torsion.
We shall now discuss whether this last property already determines
if the fibration is quasi-elliptic.

If the quasi-elliptic fibration from Theorem \ref{Thm:PSS} is jacobian,
then condition $(iii)$ requires that the fibration is extremal.
In general this means that the Picard number is maximal (relative
to the inequality \eqref{eq:Lef} or \eqref{eq:Igusa} depending
on the characteristic) while the \MoW\ group is finite.

Extremal elliptic surfaces are much more special in positive characteristic
than in characteristic zero.
In fact, Ito showed that in characteristic $p$ extremal elliptic
surfaces do always arise through purely inseparable base change
from rational elliptic surfaces \cite{Ito2}.
(Thus they are again unirational.)
Going through all extremal rational elliptic surfaces and their
purely inseparable base changes,
one can thus deduce the following solution to the above problem:

\begin{Proposition}
\label{Prop}
\label{prop:qe}
Let $X$ be a jacobian genus~1 fibration of a supersingular surface in characteristic~$2$.
If the \MoW\ group of the fibration is \hbox{$2$-elementary}
then $X$\/ is either a rational elliptic surface or quasi-elliptic.
\end{Proposition}

\begin{Remark}
\label{Rem:3}
In characteristic $3$, an analogous classification holds true
with one series of surfaces added:
elliptic surfaces with exactly two singular fibers, one of them
of type
$I_{3^e}$ for some $e\in\N$ and the other of type $II$ if $e$ is
even, or $IV^*$ if $e$ is odd (with wild ramification of index
one).
These surfaces arise from the rational elliptic surface
$y^2 +xy+tx=x^3$ through the purely inseparable base change
$t\mapsto t^{3^e}$.
Note that 
these elliptic fibrations are easy
to distinguish from quasi-elliptic fibrations thanks to the multiplicative
fiber at $t=0$.
\end{Remark}

\begin{Theorem}
\label{thm:e-qe}
Let $X$ be a K3 surface over an algebraically closed field of characteristic $p$.
Then a given jacobian genus~1 fibration on $X$ is quasi-elliptic
iff $p=2,3$, $X$ is supersingular and $\MW=(\Z/p\Z)^r$ for some $r\in\N$.
\end{Theorem}

\begin{proof}
Quasi-elliptic fibrations only occur in the specified characteristics.
For $p=2$, the theorem follows from Proposition \ref{prop:qe}.
For $p=3$, we also have to take into account the extra case from Remark \ref{Rem:3}.
But this series of surfaces avoids K3 surfaces
by inspection of the Euler-Poincar\'e characteristic, so the claim follows.
\end{proof}

The theorem (as well as the preceeding proposition) 
is useful from the lattice theoretic viewpoint
for the following reason:
As we have seen in the previous section, a jacobian genus~$1$ fibration
on an algebraic surface $X$ corresponds to a decomposition
of the \NeS\ lattice $\NS(X)=U+M$.
Here $M$ is often called the essential lattice.
If $\chi(\OO_X)>1$, then $M$
together with its root type determines the structure of 
the singular fibers and the \MoW\
group \cite{ShMW}.
Since a K3 surface has $\chi=2$,
we can thus deduce from the essential lattice $M$ whether a given jacobian
genus~$1$ fibration on a K3 surface in characteristic $2$ or $3$ is elliptic
or quasi-elliptic.
%

\section{Kneser-Nishiyama method}
\label{s:Nishi}

In \cite{Nishi}, Nishiyama introduced a lattice theoretic approach
to classify all jacobian elliptic fibrations on a complex (elliptic)
K3 surface.
The method is based on gluing techniques of Kneser and Witt \cite{Kneser}
and the classification of Niemeier lattices, i.e.~
negative-definite unimodular lattices of rank $24$.
By \cite{Nie}, there are 24 such lattices, and each is determined
by its root type.
In fact, except for the Leech lattice, the root type has always
finite index in the unimodular lattice.

For a complex K3 surface $X$, one has $\NS(X)$ of rank $\rho(X)\leq
20$.
The transcendental lattice $T(X)$ is defined as the orthogonal
complement of $\NS(X)$ in $H^2(X, \Z)$ with respect to cup-product:
\[
T(X) = \NS(X)^\bot \subset H^2(X,\Z).
\]
Since $H^2(X,\Z)$ has signature $(3,19)$, the signature of $T(X)$
is $(2, 20-\rho(X))$.
The information how to glue together $\NS(X)$ and $T(X)$ in the
unimodular lattice $H^2(X,\Z)$ is encoded in the isomorphism of
the discriminant forms:
\[
q_{\NS(X)}\0 \cong -q_{T(X)}\0.
\]
One now looks for a partner lattice $L$ of $T(X)$ with rank $26-\rho(X)$
such that $L$ is negative definite of discriminant form $q_L=q_{T(X)}$.
Such a lattice exists by lattice theory
\emph{\`{a} la} Nikulin (cf.~\cite{Nishi-Saitama}).
Then one determines all primitive embeddings of $L$ into Niemeier
lattices $N$.
For each embedding $L\hookrightarrow N$, the orthogonal complement
$M=L^\bot\subset N$ is a candidate for the essential lattice of
a jacobian elliptic fibration on $X$.

To show that $X$ does indeed admit an elliptic fibration with essential
lattice $M$, one notes that by construction the lattices $\NS(X)$
and $U+M$ have the same signature and discriminant form.
Thanks to the copy of the hyperbolic plane, these conditions imply
that the lattices are isomorphic.
But then the representation of $\NS(X)$ as $U+M$ induces a jacobian
elliptic fibration on $X$ with essential lattice $M$, as we explained
in Section \ref{s:g=1}.

Note that the same approach is not guaranteed to work in characteristic
$p>0$.
Indeed, consider supersingular K3 surfaces of Artin invariant $\sigma>2$.
Here $\NS(X)$ is $p$-elementary; hence its discriminant group has
length $2\sigma$.
Assume that $\NS(X)=U+M$, and that $M$ is embedded primitively
into some unimodular lattice $N$.
Then the discriminant group $G_L$ of its orthogonal complement
$L$ has the same length $2\sigma$.
In particular we can estimate the rank of $N$ by
\[
\mbox{rank}(N) = \mbox{rank}(M) + \mbox{rank}(L) \geq \mbox{rank}(M)
+ \mbox{length}(G_L) = 20 + 2\sigma >24.
\]
However, we can still try to pursue the same approach
for supersingular K3 surfaces with Artin invariant $\sigma\leq
2$.
This only requires to find a suitable partner lattice $L$ for $\NS(X)$.
In the present situation,
we have already mentioned that one way to write $\NS(X)$ is
$\NS(X) = U \oplus D_{20}$.
Hence we can choose $L=D_4$.
In fact, the Niemeier lattice with root system $D_{24}$ contains
 $D_4$ and $D_{20}$ as primitive orthogonal sublattices.
With the partner lattice $D_4$, we can now classify all 
genus~$1$ fibrations on $X$ (automatically jacobian by Theorem \ref{thm:jac})
and decide whether they are elliptic
or quasi-elliptic by Theorem \ref{thm:e-qe}.

Note that by Theorem \ref{thm:e-qe} it will be immediately clear from the embedding of $D_4$
into the Niemeier lattice whether the resulting genus~$1$ fibration
has non-torsion sections (and thus is elliptic).
Namely $D_4$ embeds into all root lattices of type $D_n (n\geq
4), E_n (n=6,7,8)$,
but not into any $A_n$.
The orthogonal complement of this embedding is always a root lattice
(and therefore corresponds to fiber components) unless the overlattice
in question is $D_5$ or $E_6$.
In the latter cases, the \MoW\ rank thus has to be positive,
equaling one resp.~two.

\section{Genus one fibrations on $X$}
\label{s:class}

This section gives the primitive embeddings of $L=D_4$ into Niemeier
lattices.
By the previous section, this  describes all  genus~$1$
fibrations on our K3 surface $X$.
The following table lists the root type $R(N)$ that characterizes
the corresponding Niemeier lattice $N$ uniquely.
The next entry is the root type $R(M)$ of the orthogonal complement
of the primitive embedding of $L=D_4$ into $N$.
Since this will serve as essential lattice $M$ of an elliptic fibration,
it encodes the reducible singular fibers.
The difference of the ranks of $R(M)$ and $M$ (the latter being
$20$) gives the $\MW$-rank.
As explained above, the $\MW$-rank is positive if and only if $D_4$
is embedded into $D_5$ or $E_6$.
By \cite{ShMW} we obtain the torsion subgroup of $\MW$ from the
primitive closure $R(M)'$ of $R(M)$ inside $\NS$:
\[
\MW(X)_\text{tor} \cong R(M)'/R(M).
\]
Then Proposition \ref{Prop} tells us whether the fibration will
be elliptic or quasi-elliptic, as indicated in the last column.

\begin{table}[ht!]
$$
\begin{array}{cccccc}
\hline
\# &   R(N)  &     R(L)  &  \mbox{rk}(\MW)  &  \mbox{Torsion} &
\text{elliptic?}\\
\hline
1   & D_4 A_5^4  &    A_5^4   &   0  &   3\times 6 & e\\
2  &  D_4^6   &     D_4^5   &   0   &  2^4 & qe\\
3 &  D_5^2 A_7^2  & D_5 A_7^2  &   1  &    8 & e\\
4   &D_6 A_9^2  & A_1^2 A_9^2  &  0  &   10 & e\\
5  &  D_6^4 &    A_1^2 D_6^3  &  0   &  2^3 & qe\\
6  & E_6 D_7 A_{11} &   D_7 A_{11}  &   2    &  4 & e\\
7  & E_6 D_7 A_{11} &  A_3 E_6 A_{11}  & 0   &   6 & e\\
8   &  E_6^4  &      E_6^3   &   2   &   3 & e\\
9   & D_8^3   &   D_4 D_8^2   &  0   &  2\times 2 & qe\\
10  & D_9 A_{15}  &    D_5 A_{15}  &   0    &  4 & e\\
11  & E_7 A_{17}  &  A_1^3 A_{17}  &  0   &   6 & e\\
12 & E_7^2 D_{10} &  A_1^3 E_7 D_{10} & 0  &   2\times 2 & qe\\
13 & E_7^2 D_{10}  &   D_6 E_7^2  &  0   &   2 & qe\\
14 &  D_{12}^2    &   D_8 D_{12}   &  0  &    2 & qe\\
15 &  E_8 D_{16}  &    D_4 D_{16}  &   0 &     2 & qe\\
16 &  E_8 D_{16}  &    D_{12} E_8  &   0  &    1 & qe\\
17  &  E_8^3    &   D_4 E_8^2 &   0   &   1 & qe\\
18  &  D_{24}     &    D_{20}    &   0  &    1 & qe\\
\hline
\end{array}
$$
\caption{Genus one fibrations on $X$}
\label{Tab:fibr}
\end{table}

A priori there is one ambiguity in the table:
the root lattice of type $A_1$ can correspond to singular fibers
of type $I_2$ or $III$.
In the present situation, this problem is solved as follows:

If the fibration is quasi-elliptic, then all singular fibers are
additive. Hence the above fibers have type $III$.

If the fibration is elliptic, then in each case involving an $A_1$
there is torsion in $\MW$ of order relatively prime to $2$.
Since fibers of type $III$ do not accommodate $\ell$-torsion sections
outside characteristic $\ell$ ($\ell\neq 2$),
the fibers corresponding to $A_1$'s have type $I_2$.

Table \ref{Tab:fibr} settles the classification statement of Theorem
\ref{thm}.
It remains to prove existence and uniqueness for each genus~$1$ fibration.
This will be achieved in Section \ref{s:eq}, as outlined in the
next section, and Section \ref{s:unique}.

\begin{Remark}
In our concrete situation, we can also distinguish elliptic and quasi-elliptic fibrations,
given a decomposition $\NS(X)=U+M$,
by computing the Euler-Poincar\'e characteristics of the singular fibers
instead of appealing to Theorem \ref{thm:e-qe}.
Since some additive fiber types on an elliptic fibration
necessarily come with wild ramification by \cite{SS2},
this in fact suffices for all cases but \#18
which is implied by \cite{S-max} to be quasi-elliptic.
\end{Remark}

Several of the fibrations from Table \ref{Tab:fibr} have been studied
by Dolgachev and Kond\=o in \cite{DK}, by Ito in \cite{Ito2}, 
 and by one of us in \cite{S-MJM},
see also
 \cite{KK}, \cite[App.~2]{OS}, \cite{RS2}, \cite{Schroeer}, \cite[Ex.~4.1]{Sh77} as indicated in the following sections.
Here we complement the previous considerations to derive equations
and connections for all fibrations.
We conclude this section with a remark about Picard numbers over
finite fields.
For each fibration, we will exhibit a model over $\F_2$ with Picard
number $22$ over $\F_4$.
However, the question of the Picard number over $\F_2$ is more
subtle.
We will see in the next section that the first two fibrations admit
models $X$ with $\rho(X/\F_2)=15$.
This cannot be improved because of the Galois action on the singular
fibers and their components (or on the \MoW\ group).
In contrast, for all other fibrations we will exhibit models with
$\rho(X/\F_2)=21$.
This is optimal for supersingular K3 surfaces by \cite[(6.8)]{Artin} 
(see also \cite[Thm.~4.4]{S-MJM}, \cite{Sss}).
More precisely, 
we will show that all models with $\rho(X/\F_2)$ fixed 
(i.e.~$15$ or $21$) are isomorphic over $\F_2$.
In order to move between these two groups,
we will exhibit two different models of $\# 5$ which are isomorphic over $\F_4=\F_2(\varrho)$
with $\varrho^2+\varrho+1=0$.

\section{Plan for connections}

Let $S$ be a projective K3 surface.
Recall that it suffices to identify a divisor $D$ on $S$
that has the shape of a singular fiber from Kodaira's list
in order to find an genus~$1$ fibration on~$S$\/ with $D$\/ as singular fiber.
The fibration is induced by the linear system $|D|$.
Moreover, any irreducible curve $C$ with $C\cdot D=1$ gives a section
of the fibration.

With these tools at hand, it is in principle possible to derive
all fibrations in Table \ref{Tab:fibr} from a single model of the
surface $X$.
In practice, however, it is often easier to pursue this aim in
several steps, since one can usually find only a few linear systems
without too much effort.
The following diagram sketches how we will connect all fibrations.
The numbers refer to the figures in the next section where the
connections are derived (or in one case to a subsection which provides
a further reference).

$$
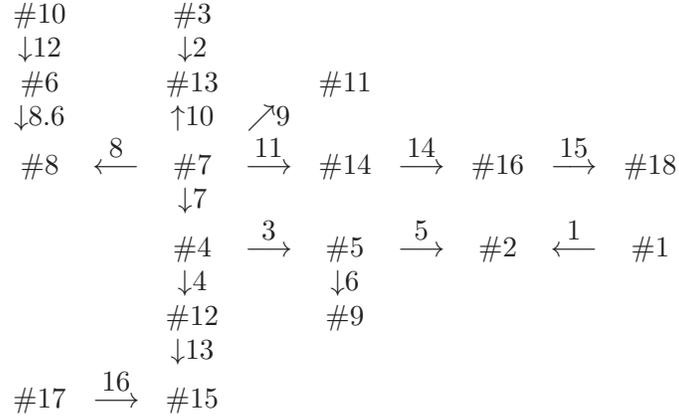
\begin{array}{ccccccccc}
 \#10 &&
\#3
&&&&&&\\
 {\downarrow} {\scriptstyle\ref{Fig:10-6}} &&
{\downarrow} {\scriptstyle \ref{Fig:3-13}}
&&&&&&\\
 \#6 && \#13 && \#11 &&&&\\
{\downarrow} {\scriptstyle\ref{ss:6}}
   && {\uparrow} {\scriptstyle \ref{Fig:7-13}}
    & {\nearrow} {\scriptstyle \ref{Fig:7-11}}
      &&&&&\\
 \#8 & \stackrel{\ref{Fig:7-8}}{\longleftarrow} & \#7 & \stackrel{\ref{Fig:7-14}}{\longrightarrow}
& \#14 & \stackrel{\ref{Fig:14-16}}{\longrightarrow} & \#16 & \stackrel{\ref{Fig:16-18}}{\longrightarrow}
& \#18\\
&&
{\downarrow} {\scriptstyle \ref{Fig:7-4}}
&&&&&&\\
&&
\#4&
\stackrel{\ref{Fig:4-5}}{\longrightarrow}
&
\#5&
\stackrel{\ref{Fig:5-2}}{\longrightarrow}
&
\#2
&
\stackrel{\ref{Fig:1-2}}{\longleftarrow}
&
\#1\\
&&
{\downarrow} {\scriptstyle \ref{Fig:4-12}}
&&
{\downarrow} {\scriptstyle \ref{Fig:5-9}}
&&&&\\
&&
\#12
&&
\#9
&&&&\\
&&
{\downarrow} {\scriptstyle \ref{Fig:12-15}}
&&&&&&\\

\#17 &
\stackrel{\ref{Fig:17-15}}{\longrightarrow}
& \#15
&&&&&&
\end{array}
$$

\section{Equations \& Connections}
\label{s:eq}

Usually we shall use affine coordinates $x,y,t$ with $t$ as the
parameter of the base curve $\PP^1$ over $\F_2$.
The new parameter will be denoted by $u$, i.e.~it exhibits a new
genus~$1$ fibration on $X$ by the surjection
\begin{eqnarray*}
X & \to & \PP^1\\
(x,y,t) & \mapsto & u(x,y,t)
\end{eqnarray*}

A 5-tuple $[a_1,a_2,a_3,a_4,a_6]$ refers to the usual short-hand
notation for the elliptic curve
\[
y^2 + a_1xy+a_3y = x^3 + a_2x^2+a_4x+a_6.
\]
This fibration is quasi-elliptic in characteristic $2$ if and only
if $a_1\equiv a_3\equiv 0$ identically.

\subsection{\#1: $R(M) = A_5^4$}

This fibration arises as inseparable base change from the Hesse
pencil (see \cite[Ex. 4.1]{Sh77}):
\begin{eqnarray}\label{eq:Hesse}
X:\;\; x^3+y^3+z^3 = t^2xyz.
\end{eqnarray}

A Weierstrass model can be found for instance in \cite{Ito2}.
We have sections at the base points of the cubics (induced from
the Hesse pencil) plus the likes of $[x,y,z]=[t,1,1]$.
In total the sections are always given by $x^3=z^3$ or $y^3=z^3$
or $x^3=y^3$.

\subsubsection*{Connection with \#2}
We can extract $\tilde D_4$ divisors from sections and fiber components.

We shall work affinely in the chart $z=1$.
For instance by setting $u=y$, we visibly arrange for $\td_4$ fibers
at $u=0,\infty$.

In the sequel we will draw figures with fiber components and sections
to visualize the connections.
We will distinguish as follows between old and new fibration:

\begin{center}
\begin{tabular}{rl}
old fiber components & balls\\
old sections & small circles\\
&\\
new fibers & framed by boxes\\
new sections & big circles
\end{tabular}
\end{center}

The center of the following figure sketches the components of the
$I_6$ fiber at $t=\infty$.
We identify the fiber components $C_x, C_y, C_z$ given by $x=0$
resp.~$y=0$ resp.~$z=0$ of the model (\ref{eq:Hesse}).
The other three components arise as the exceptional divisors above
the singular points at their intersections.
The given sections come from the base points of the Hesse pencil
with $y=0, x^3=z^3$ (LHS) or $x^3=y^3, z=0$ (RHS).
The component $C_x$ serves as a section of the new fibration.

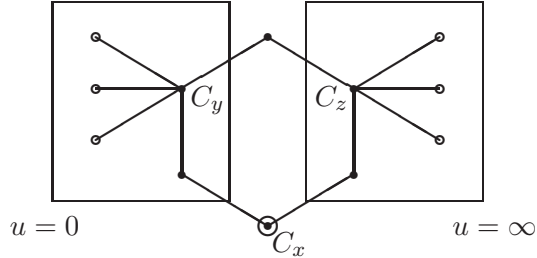
\begin{figure}[ht!]
\setlength{\unitlength}{.45in}
\begin{picture}(10,3.2)(0,0)
\thicklines


\multiput(4,2)(2,0){2}{\circle*{.1}}
\multiput(4,1)(2,0){2}{\circle*{.1}}
\put(4,2){\line(0,-1){1}}
\put(6,2){\line(0,-1){1}}
\put(5,0.4){\circle*{.1}}
\put(5,2.6){\circle*{.1}}

\put(5,0.4){\circle{0.2}}

\put(3,1.4){\line(5,3){2}}
\put(5,0.4){\line(5,3){1}}
\put(4,1){\line(5,-3){1}}
\put(5,2.6){\line(5,-3){2}}

\put(3,1.4){\circle{.1}}
\put(3,2){\circle{.1}}
\put(3,2.6){\circle{.1}}
\put(3,2){\line(1,0){1}}

\put(3,2.6){\line(5,-3){1}}
\put(6,2){\line(5,3){1}}

\put(7,1.4){\circle{.1}}
\put(7,2){\circle{.1}}
\put(7,2.6){\circle{.1}}
\put(6,2){\line(1,0){1}}

\put(2,0.3){$u=0$}
\put(7.15,0.3){$u=\infty$}
\put(5.05,0.1){$C_x$}
\put(4.1,1.8){$C_y$}
\put(5.55,1.8){$C_z$}

%
\thinlines
\put(2.5,0.7){\framebox(2.05,2.3){}}
\put(5.45,0.7){\framebox(2.05,2.3){}}

\end{picture}
\caption{Two $\tilde D_4$ divisors supported on $\tilde A_5$ and
sections}
\label{Fig:1-2}
\end{figure}

This yields the quasi-elliptic fibration
\[
X:\;\; t^2 = ux(x^3+u^3+1).
\]
This  can be transformed into Weierstrass form as follows. First
homogenize the RHS as a quartic polynomial with variable $z$. Setting
$x=1$, we obtain a cubic in Weierstrass form up to some factors:
\[
X:\;\; t^2 = u((u^3+1)z^3+1).
\]
The change of variables $(z,t)\mapsto (z/(u(u^3+1))^2, t/(u(u^3+1))^2)$
then returns the Weierstrass form
\[
X:\;\; t^2 = z^3 + u^3(u^3+1)^2.
\]
One reads off singular fibers of type $\tilde D_4$ at $u=0, \infty$
(as seen above) and at the roots of $u^3+1$.

\subsection{\#2: $R(M)=D_4^5$}
\label{ss:2}

This fibration admits several nice models, for instance $[0,0,0,0,(t^3+1)^3]$
with singular fibers at all points of $\PP^1(\F_4)$ as seen above. 
There are plenty of automorphisms respecting the fibration, for instance
\[
\alpha: (x,y,t)\mapsto (\varrho x,y,t)
\]
for $\varrho^3=1$ and those induced by M\"obius transformations
of $\PP^1$ that permute $\infty$ and third roots of unity such
as
\[
(x,y,t) \mapsto (x/(t+1)^4, y/(t+1)^6, t/(t+1)).
\]
$\MW=(\Z/2\Z)^4$ with sections $P=(t^3+1,0), Q=(t(t^3+1), (t^3+1)^2)$
plus images under above automorphisms.

As an example, we give two connections, but we shall not use them here,
since they do not lead to models with maximal Picard number $\rho(X/\F_2)=21$
although the new fibrations admit such models (cf.~\ref{s:5}).
In the sequel, we shall only give the connections needed for the
proof of Theorem \ref{thm}.

\subsubsection*{Connection with \#3}
$u=y/((t^2+t+1)(x+t^3+1))$ extracts (independently at $u=0$ and
$\infty$) two $\tilde A_7$ divisors from pairs of two $\tilde D_4$
fibers connected through two sections.

%

%



%









\subsubsection*{Connection with \#8}
$u=y/(t^3+1)^2$ extracts $\te_6$ from $\td_4$ at $\infty$ and two-torsion
sections $P, \alpha P, \alpha^2 P$ at $u=0$.
Same at $u=\infty$ from zero-section plus identity and double components
of $\td_4$ fibers at roots of  $t^3+1$.
The remaining simple components of the fibers at the roots of $t^3+1$
serve as sections.

%

%











\subsection{\#3: $R(M) = D_5A_7^2$}

From \#2, we can obtain the model of \#3 as cubic pencil
\[
X:\;\; (x^2+x+1)(y+1) = u^2 (y^2+y+1)(x+1).
\]
This fibration is a purely inseparable base change by $s=u^2$ from
a rational elliptic surface $S$ with configuration $[4,4,III]$.
Here the $III$-fiber comes with wild ramification of index one;
since the ramification index stays constant under the base change,
the special fiber is replaced by type $I_1^*$ as claimed.
The base points of the pencil generate $\MW(S) \cong \Z \times
\Z/4\Z$. 

We find  generators of $\MW(X)$ in terms of another model of this
elliptic fibration
which also has the advantage of maximal Picard number $\rho(X/\F_2)=21$.
It arises from the extremal rational elliptic surface $[1,s^2,s^2,0,0]$
with singular fibers of type $I_8$ at $t=0$ and $III$ at $\infty$
through the base change $t\mapsto s=t^2+t$:
\begin{eqnarray}
\label{eq3}
X:\;\; y^2 + xy + (t^2+t)^2 y = x^3 + (t^2+t)^2 x^2.
\end{eqnarray}
Next to the induced torsion sections $(s^2,0), (0,0), (0,s^2)$,
there is an $8$-torsion section $P=(t^2(t+1), t^4 (t+1))$.
Moreover there is an induced section $Q=(t^2,\varrho t^4)$ of height $1$. 
By computing the discriminant of $\NS(X)$, one verifies that these
sections generate $\MW(X)$.


\subsubsection*{Connection with \#13}

$u=(x+s^2)/s^4$ extracts an $\tilde E_7$ at $u=\infty$
from the $\tilde A_7$ at $s=0$ and the zero section.
The non-identity components of the other $\tilde A_7$ together
with the two-torsion section $Q=(0,0)$ form another $\tilde E_7$
at $u=1$.
This leaves a root lattice $D_5$ ($\tilde D_5$ minus identity component)
at $\infty$ disjoint;
on the new fibration it results in a singular fiber of type  $\tilde
D_6$ at $u=0$.
As a new section, one can take $P$.

\begin{figure}[ht!]
\setlength{\unitlength}{.45in}
\begin{picture}(10.5,5)(-0.25,0.5)
\thicklines

%

\put(5.5,3){\circle*{.1}}
\put(6,4){\circle*{.1}}
\put(7,4.5){\circle*{.1}}
\put(8,4){\circle*{.1}}

\put(8.5,3){\circle*{.1}}

\put(5.5,3){\line(1,2){.5}}
\put(6,4){\line(2,1){1}}

\put(8.5,3){\line(-1,2){.5}}
\put(8,4){\line(-2,1){1}}

\put(6,2){\circle*{.1}}
\put(7,1.5){\circle*{.1}}
\put(8,2){\circle*{.1}}

\put(5.5,3){\line(1,-2){.5}}
\put(6,2){\line(2,-1){1}}

\put(8.5,3){\line(-1,-2){.5}}
\put(8,2){\line(-2,-1){1}}

\put(0.5,3){\circle*{.1}}
\put(1,4){\circle*{.1}}
\put(2,4.5){\circle*{.1}}
\put(3,4){\circle*{.1}}

\put(8.5,3){\line(1,0){1.5}}
\put(9.5,3){\circle{.1}}
\put(9.15,3.15){$Q$}





\put(3.5,3){\circle*{.1}}

\put(.5,3){\line(1,2){.5}}
\put(1,4){\line(2,1){1}}

\put(3.5,3){\line(-1,2){.5}}
\put(3,4){\line(-2,1){1}}

\put(1,2){\circle*{.1}}
\put(2,1.5){\circle*{.1}}
\put(3,2){\circle*{.1}}


\put(.5,3){\line(1,-2){.5}}
\put(1,2){\line(2,-1){1}}

\put(3.5,3){\line(-1,-2){.5}}
\put(3,2){\line(-2,-1){1}}

\put(4.5,3){\circle{.1}}
\put(3.5,3){\line(1,0){2}}
\put(4.15,3.15){$O$}

\put(0,3){\line(1,0){.5}}

\put(5.25,5){\circle{.1}}
\put(5.25,5){\circle{.2}}
\put(5.25,5){\line(3,-4){.75}}
\put(5.4,5.1){$P$}

\qbezier(1,4)(1,6)(5.25,5)








\thinlines

\put(.8,0.6){$u=\infty$}
\put(8.8,.6){$u=1$}

\put(0.75,1){\framebox(4,4){}}
\put(5.75,1){\framebox(4,4){}}

\end{picture}
\caption{Two $\tilde E_7$ divisors supported on two $\tilde A_7$'s
and sections}
\label{Fig:3-13}
\end{figure}
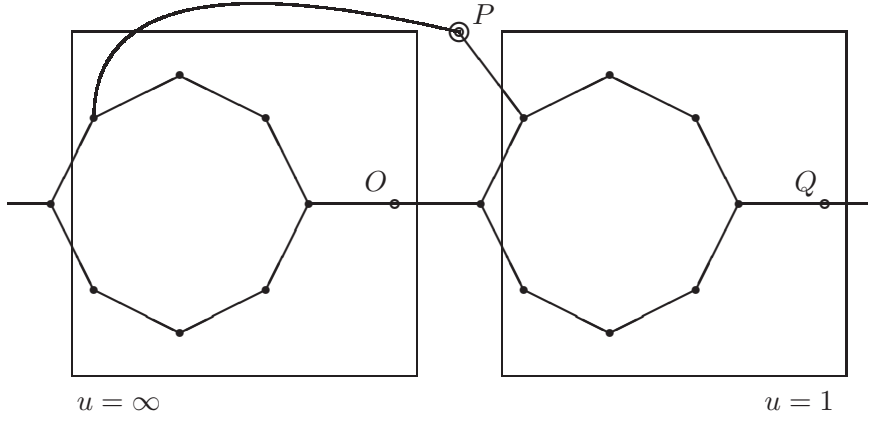

We take this example as an opportunity to explain how one can derive
the Weierstrass form of the new fibration explicitly.
In general, it is often instructive to work with some resolution
of singularities related to the new coordinate $u$.
Here it concerns  the $A_7$  singularity
 of the Weierstrass form \eqref{eq3} at $(x,y,s)=(0,0,0)$.
We proceed in two steps, always choosing an appropriate affine
chart.
Blowing up twice yields affine coordinates
\[
x=s^2x'', y=s^2y''.
\]
The Weierstrass form transforms as
\begin{eqnarray}
\label{eq3-1}
X:\;\; y''^2 + x''y'' + (s+1)^2 y'' = s^2x''^3 + (s^2+s)^2 x''^2.
\end{eqnarray}
Here the section $P$ takes the shape $(x'', y'')=(s+1, s^2(s+1))$.
The node of the above fibration in the fiber $s=0$ sits at $(x'',y'')=(1,0)$.
Hence we translate $x''$ by $1$ and then blow-up two more times.
This brings us exactly to the coordinate $u$ from above (and another
coordinate $v$):
\[
x''=s^2u+1, y''=s^2v.
\]
Here \eqref{eq3-1} transforms as
\begin{eqnarray}
\label{eq3-2}
X:\;\; v^2+uv+v = s^4u^3+s^4u^2+u+1.
\end{eqnarray}
The section $P$ is expressed as $(u,v)=(1/s, s+1)$.
Now we want to consider \eqref{eq3-2} as an elliptic fibration
onto $u\in\PP^1$.
Then $P$ gives us the section $(s,v)=(1/u, 1+1/u)$.
In order to obtain a Weierstrass form, we first translate $s$ and
$v$ by the coordinates of the  section.
This gives
\[
X:\;\; v^2+(u+1)v = u^2(u+1)s^4.
\]
We now modify $v\mapsto sv$, yielding the following plane cubic
\[
X:\;\; sv^2+(u+1)v = u^2(u+1)s^3.
\]
Next we homogenize by the variable $w$ and set $v=1$ to obtain
the following quasi-elliptic fibration:
\[
X:\;\; (u+1)w^2 = u^2(u+1)s^3+s.
\]
Finally the variable change $(s,w)\mapsto(s/(u(u+1))^2, w/(u^2(u+1)^3)$
gives the Weierstrass form
\[
X:\;\; w^2 = s^3+u^2(u+1)^3s.
\]
One immediately checks that this has singular fibers of type $\tilde
D_6$ at $u=0$ and $\tilde E_7$ at $u=1,\infty$ as predicted.
Similar computations apply to all other connections.

\subsection{\#4: $R(M) = A_1^2 A_9^2$}
This fibration arises from (the mod $2$ reduction of) the universal
elliptic curve for $\Gamma_1(5)$ by purely inseparable base change.
A model can be given as $[t^2+1, t^2, t^2, 0, 0]$ with $\tilde
A_9$'s at $0,\infty$ and $\tilde A_1$'s at the roots of $t^2+t+1$.
$\MW=\Z/10\Z$ with  $5$-torsion section induced from the universal
elliptic  curve, generated by $(0,0)$ or $(t^2,0)$ for instance.
As an extra feature there is a $2$-torsion section $(t^2/(t+1)^2,
t^4/(t+1)^3)$ meeting the zero section. (This can only happen for
$p^n$-torsion in characteristic $p$; Shioda calls such torsion
sections peculiar in \cite{OS}).
Sections of order ten are e.g.~$P=(t,t)$  and $(t^2+t^3, t^4)$.


\subsubsection*{Connection with \#5}
$u=x/t^2$ extracts $\tilde D_6$ from $\tilde A_9$'s and zero section.
The remaining fiber components combine with sections $4P, 6P$ (at
$t=0$) resp.~$2P, 8P$ (at $t=\infty$) for two further copies of
$\tilde D_6$.
$A_1$'s stay unchanged.

The eight two-torsion sections of the new fibration come from the
remaining four fiber components of the two $\tilde A_9$ fibers
and the four ten-torsion sections $P, 3P, 7P, 9P$.

\begin{figure}[ht!]
\setlength{\unitlength}{.45in}
\begin{picture}(10,5.5)(-0.25,0.5)
\thicklines

%

\put(6,3){\circle*{.1}}
\put(6.25,4){\circle*{.1}}
\put(7.25,4.5){\circle*{.1}}
\put(7.25,4.5){\circle{.2}}
\put(8.25,4.5){\circle*{.1}}

\put(9.5,3){\circle*{.1}}
\put(9.25,4){\circle*{.1}}

\put(6,3){\line(1,4){.25}}
\put(6.25,4){\line(2,1){1}}
\put(7.25,4.5){\line(1,0){1}}

\put(9.5,3){\line(-1,4){.25}}
\put(9.25,4){\line(-2,1){1}}

\put(6.25,2){\circle*{.1}}
\put(7.25,1.5){\circle*{.1}}
\put(7.25,1.5){\circle{.2}}
\put(8.25,1.5){\circle*{.1}}

\put(9.25,2){\circle*{.1}}

\put(6,3){\line(1,-4){.25}}
\put(6.25,2){\line(2,-1){1}}
\put(7.25,1.5){\line(1,0){1}}

\put(9.5,3){\line(-1,-4){.25}}
\put(9.25,2){\line(-2,-1){1}}
\put(2.75,4.5){\circle{.2}}
\put(2.75,1.5){\circle{.2}}

\put(0.5,3){\circle*{.1}}
\put(0.75,4){\circle*{.1}}
\put(1.75,4.5){\circle*{.1}}
\put(2.75,4.5){\circle*{.1}}

\put(.75,4){\line(0,1){.75}}
\put(.74,4.75){\circle{.1}}
\put(.9,4.75){$6P$}

\qbezier(.75,4.75)(-3.5,-2)(7.25,1.5)

\put(.75,2){\line(0,-1){.75}}
\put(.75,1.25){\circle{.1}}
\put(.9,1.05){$4P$}

\qbezier(.75,1.25)(-3.5,8)(7.25,4.5)

\put(4,3){\circle*{.1}}
\put(3.75,4){\circle*{.1}}

\put(.5,3){\line(1,4){.25}}
\put(.75,4){\line(2,1){1}}
\put(1.75,4.5){\line(1,0){1}}

\put(4,3){\line(-1,4){.25}}
\put(3.75,4){\line(-2,1){1}}

\put(.75,2){\circle*{.1}}
\put(1.75,1.5){\circle*{.1}}
\put(2.75,1.5){\circle*{.1}}

\put(3.75,2){\circle*{.1}}

\put(.5,3){\line(1,-4){.25}}
\put(.75,2){\line(2,-1){1}}
\put(1.75,1.5){\line(1,0){1}}

\put(4,3){\line(-1,-4){.25}}
\put(3.75,2){\line(-2,-1){1}}

\put(5,3){\circle{.1}}
\put(4,3){\line(1,0){2}}
\put(5.1,3.15){$O$}

\put(8.25,1){\circle{.1}}
\put(9.25,2){\line(-1,-1){1}}
\put(8.4,0.7){$2P$}

\qbezier(2.75,1.5)(7,0)(8.25,1)

\put(8.25,5){\circle{.1}}
\put(9.25,4){\line(-1,1){1}}
\put(8.4,5){$8P$}

\qbezier(2.75,4.5)(7,6)(8.25,5)






\thinlines

\put(5,1.2){$u=\infty$}

\put(3.5,1.5){\framebox(3,3){}}
\put(0,.75){\framebox(2.1,4.5){}}
\put(8,.5){\framebox(2.1,5){}}

\end{picture}
\caption{Three $\tilde D_6$ divisors supported on two $\tilde A_9$'s
and sections}
\label{Fig:4-5}
\end{figure}
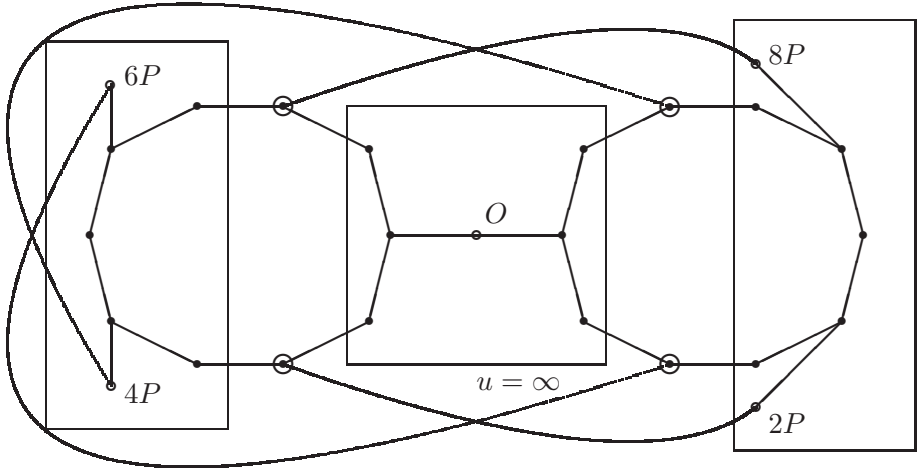

%


\subsubsection*{Connection with \#12}
$u=x$ extracts $\tilde E_7$ from $\tilde A_9$ at $\infty$ and zero
section.
Non-identity components of $\tilde A_9$ at $t=0$ and sections $2P,
8P$ form $\tilde D_{10}$;
the two $A_1$'s formed by the non-identity fiber components at
roots of $t^2+t+1$ remain, and  there is another $A_1$ given by
the opposite component of the $\tilde A_9$ at $\infty$.
The sections of \#12 are thus given by the two fiber components
indicated in the figure, and by the old sections $P, 9P$.

\begin{figure}[ht!]
\setlength{\unitlength}{.45in}
\begin{picture}(10,5)(-0.25,0.5)
\thicklines

%

\put(6,3){\circle*{.1}}
\put(6.25,4){\circle*{.1}}
\put(7.25,4.5){\circle*{.1}}
\put(8.25,4.5){\circle*{.1}}

\put(9.5,3){\circle*{.1}}
\put(9.25,4){\circle*{.1}}

\put(6,3){\line(1,4){.25}}
\put(6.25,4){\line(2,1){1}}
\put(7.25,4.5){\line(1,0){1}}

\put(9.5,3){\line(-1,4){.25}}
\put(9.25,4){\line(-2,1){1}}

\put(6.25,2){\circle*{.1}}
\put(7.25,1.5){\circle*{.1}}
\put(8.25,1.5){\circle*{.1}}

\put(9.25,2){\circle*{.1}}
\put(9.25,2){\circle{.2}}
\put(9.25,4){\circle{.2}}

\put(6,3){\line(1,-4){.25}}
\put(6.25,2){\line(2,-1){1}}
\put(7.25,1.5){\line(1,0){1}}

\put(9.5,3){\line(-1,-4){.25}}
\put(9.25,2){\line(-2,-1){1}}

\put(0.5,3){\circle*{.1}}
\put(0.75,4){\circle*{.1}}
\put(1.75,4.5){\circle*{.1}}
\put(2.75,4.5){\circle*{.1}}

\put(4,3){\circle*{.1}}
\put(3.75,4){\circle*{.1}}

\put(.5,3){\line(1,4){.25}}
\put(.75,4){\line(2,1){1}}
\put(1.75,4.5){\line(1,0){1}}

\put(4,3){\line(-1,4){.25}}
\put(3.75,4){\line(-2,1){1}}

\put(.75,2){\circle*{.1}}
\put(1.75,1.5){\circle*{.1}}
\put(2.75,1.5){\circle*{.1}}

\put(3.75,2){\circle*{.1}}

\put(.5,3){\line(1,-4){.25}}
\put(.75,2){\line(2,-1){1}}
\put(1.75,1.5){\line(1,0){1}}

\put(4,3){\line(-1,-4){.25}}
\put(3.75,2){\line(-2,-1){1}}

\put(5,3){\circle{.1}}
\put(4,3){\line(1,0){2}}
\put(5.1,3.15){$O$}

\put(3.25,1){\circle{.1}}
\put(3.25,1){\line(-1,1){.5}}
\put(3.3,1.15){$2P$}

\qbezier(3.25,1)(9,0)(9.25,2)

\put(3.25,5){\circle{.1}}
\put(3.25,5){\line(-1,-1){.5}}
\put(3.3,4.7){$8P$}

\qbezier(3.25,5)(9,6)(9.25,4)






\thinlines

\put(4.85,1.4){$u=\infty$}
\put(.4,.9){$u=0$}

\put(4.75,1.25){\framebox(3.75,3.5){}}
\put(0.25,.75){\framebox(3.65,4.5){}}

\end{picture}
\caption{$\tilde E_7$ and $\tilde D_{10}$ divisors supported on
two $\tilde A_9$'s and sections}
\label{Fig:4-12}
\end{figure}
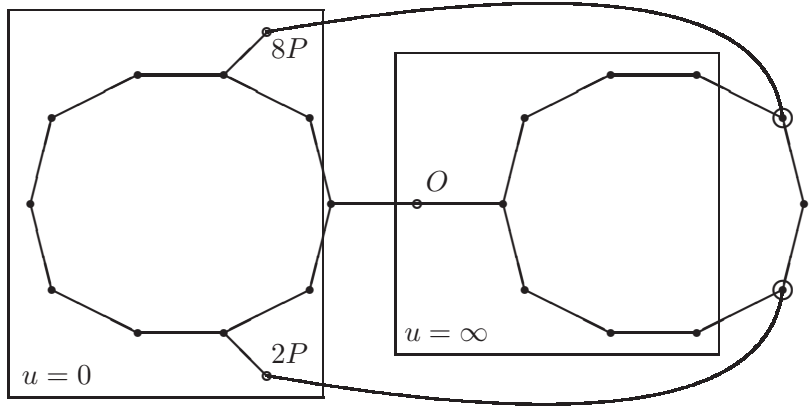

\subsection{\#5: $R(M) = A_1^2 D_6^3$}
\label{s:5}

For this quasi-elliptic fibration, we shall exhibit two models
in order to transfer from the models with $\rho(X/\F_2)=15$ (\#'s
1, 2) to all other fibrations with optimal models of $\rho(X/\F_2)=21$.
We start with the quasi-elliptic fibration $[0,0,0,t(t^3+1)^2,0]$
with $\tilde D_6$'s at roots of $t^3+1$ and $\ta_1$'s at $t=0,\infty$.
This model has $\rho(X/\F_2)=15$:
from $\rho(X/\F_4)=22$,
we first have to subtract $6$ divisors for the two $\tilde D_6$
that are conjugate over $\F_4$.
By Tate's algorithm, the far components of the $\tilde D_6$ at
$t=1$ are also conjugate over $\F_4$.
This accounts for the seventh divisor 
which is not Galois invariant over $\F_2$.

$\MW\cong (\Z/2\Z)^3$ with sections $P=(0,0), ((t+1)(t^3+1),(t^3+1)^2),
Q=((t^2+t+1)t, (t^2+t+1)^2t)$ and their images under the automorphism
$(x,y,t)\mapsto (\varrho x, y, \varrho^2 t)$.

\subsubsection*{Connection with \#2}
$u=x/(t^3+1)$ extracts two $\tilde D_4$'s from identity components
of $\td_6$'s and $\ta_1$ at $\infty$ plus
zero section (at $u=\infty$) or from the section $P$ and the fiber
components outside $t=\infty$ meeting it (at $u=0$).
As new sections, we derive some double fiber components as depicted
in the figure.
Not that one of them is indeed defined over $\F_2$.

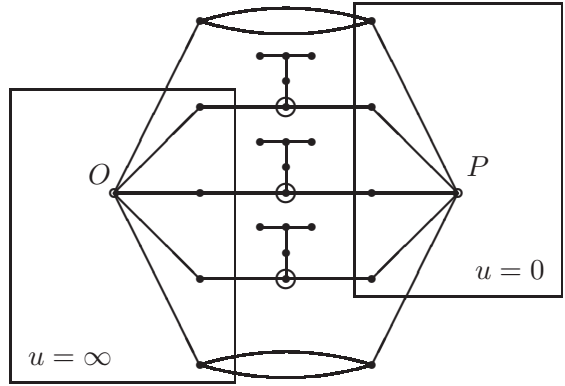
\begin{figure}[ht!]
\setlength{\unitlength}{.45in}
\begin{picture}(10,5.2)(0,-0.5)
\thicklines


\multiput(3,2)(4,0){2}{\circle{.1}}
\multiput(4,2)(1,0){3}{\circle*{.1}}
\multiput(4,1)(1,0){3}{\circle*{.1}}
\multiput(4,3)(1,0){3}{\circle*{.1}}
\multiput(4,4)(2,0){2}{\circle*{.1}}
\multiput(4,0)(2,0){2}{\circle*{.1}}

\multiput(5,3.3)(0,-1){3}{\circle*{.1}}
\multiput(5,3.6)(0,-1){3}{\circle*{.1}}

\multiput(5.3,3.6)(0,-1){3}{\circle*{.1}}
\multiput(4.7,3.6)(0,-1){3}{\circle*{.1}}

\multiput(5,3.6)(0,-1){3}{\line(0,-1){.6}}

\multiput(4.7,3.6)(0,-1){3}{\line(1,0){.6}}

\put(3,2){\line(1,2){1}}
\put(3,2){\line(1,1){1}}
\put(3,2){\line(1,-1){1}}
\put(3,2){\line(1,-2){1}}

\put(7,2){\line(-1,2){1}}
\put(7,2){\line(-1,-2){1}}
\put(7,2){\line(-1,1){1}}
\put(7,2){\line(-1,-1){1}}

\qbezier(4,0)(5,.3)(6,0)
\qbezier(4,0)(5,-0.3)(6,0)

\qbezier(4,4)(5,4.3)(6,4)
\qbezier(4,4)(5,3.7)(6,4)


\put(7.1,2.2){$P$}

\put(2.7,2.1){$O$}

\multiput(5,3)(0,-1){3}{\circle{.2}}


%
\put(3,2){\line(1,0){4}}
\put(4,3){\line(1,0){2}}
\put(4,1){\line(1,0){2}}

\thinlines
\put(5.8,0.8){\framebox(2.4,3.4){}}
\put(1.8,-0.2){\framebox(2.6,3.4){}}

\put(2,0){$u=\infty$}
\put(7.2,1){$u=0$}


\end{picture}
\caption{Two $\tilde D_4$ fibers supported on three $\tilde D_6$'s,
two $\tilde A_1$'s and two sections}
\label{Fig:5-2}
\end{figure}

In order to connect with \#9, we exhibit another model of this
fibration that admits the maximal Picard number $\rho(X/\F_2)=21$.
The coordinate change
\begin{eqnarray}
\label{eq:5}
\;\;\;\;\;\;
(x,y,t) \mapsto (\varrho^2 x/(t+1+\varrho^2)^2, y/(t+1+\varrho^2)^3,
\varrho (t+1+\varrho)/(t+1+\varrho^2)
\end{eqnarray}
yields the quasi-elliptic fibration $[0,0,0,t^2(t+1)^2(t^2+t+1),0]$.
One easily verifies that the $\tilde D_6$ fibers have all components
defined over $\F_2$,
so $\rho(X/\F_2)=21$.

\subsubsection*{Connection with \#9}
$u=x/((t^2+t+1)t)$ extracts $\tilde D_4$ from zero section and
identity components of $\tilde A_1$'s and $\tilde D_6$'s at $0$
and $\infty$.
There are two disjoint copies of $\tilde D_8$.
One involves most of the $\tilde D_6$ at $t=1$ as in the figure;
the other connects the two $\tilde D_6$ at $0$ and $\infty$ by
the section $Q$ .
In the new coordinates of \eqref{eq:5}, this section reads $Q=(t(t^2+t+1),t^2(t^2+t+1))$.

As new torsion sections, we identify the two fiber components depicted
in the figure, and the two old sections
$((t+1)(t^2+t+1),(t+1)^2(t^2+t+1))$ and $(t(t+1)(t^2+t+1),t^2(t+1)^2(t^2+t+1))$.

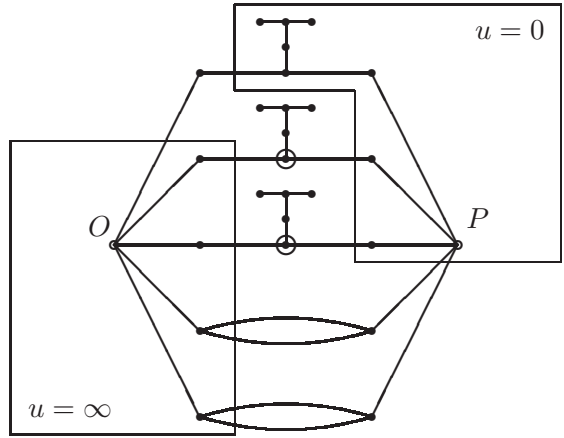
\begin{figure}[ht!]
\setlength{\unitlength}{.45in}
\begin{picture}(10,5.6)(0,-0.5)
\thicklines


\multiput(3,2)(4,0){2}{\circle{.1}}
\multiput(4,2)(1,0){3}{\circle*{.1}}
\multiput(4,4)(1,0){3}{\circle*{.1}}
\multiput(4,3)(1,0){3}{\circle*{.1}}
\multiput(4,1)(2,0){2}{\circle*{.1}}
\multiput(4,0)(2,0){2}{\circle*{.1}}

\multiput(5,4.3)(0,-1){3}{\circle*{.1}}
\multiput(5,4.6)(0,-1){3}{\circle*{.1}}

\multiput(5.3,4.6)(0,-1){3}{\circle*{.1}}
\multiput(4.7,4.6)(0,-1){3}{\circle*{.1}}

\multiput(5,4.6)(0,-1){3}{\line(0,-1){.6}}

\multiput(4.7,4.6)(0,-1){3}{\line(1,0){.6}}

\put(3,2){\line(1,2){1}}
\put(3,2){\line(1,1){1}}
\put(3,2){\line(1,-1){1}}
\put(3,2){\line(1,-2){1}}

\put(7,2){\line(-1,2){1}}
\put(7,2){\line(-1,-2){1}}
\put(7,2){\line(-1,1){1}}
\put(7,2){\line(-1,-1){1}}

\qbezier(4,0)(5,.3)(6,0)
\qbezier(4,0)(5,-0.3)(6,0)

\qbezier(4,1)(5,1.3)(6,1)
\qbezier(4,1)(5,0.7)(6,1)


\put(7.1,2.2){$P$}

\put(2.7,2.1){$O$}


%
\put(3,2){\line(1,0){4}}
\put(4,3){\line(1,0){2}}
\put(4,4){\line(1,0){2}}

\multiput(5,3)(0,-1){2}{\circle{.2}}

\thinlines
\put(1.8,-0.2){\framebox(2.6,3.4){}}

\put(5.8,1.8){\line(0,1){2}}
\put(5.8,1.8){\line(1,0){2.4}}
\put(8.2,1.8){\line(0,1){3}}
\put(5.8,3.8){\line(-1,0){1.4}}
\put(4.4,3.8){\line(0,1){1}}
\put(4.4,4.8){\line(1,0){3.8}}

\put(2,0){$u=\infty$}
\put(7.2,4.4){$u=0$}


\end{picture}
\caption{$\tilde D_4$ and $\tilde D_8$ supported on three $\tilde
D_6$'s, two $\tilde A_1$'s and two sections}
\label{Fig:5-9}
\end{figure}

\subsection{\#6: $R(M) = D_7A_{11}$}
\label{ss:6}

Elliptic fibration given by $[1,t^3,t^3,0,0]$ with $\ta_{11}$ at
$t=0$ and $\td_7$ at $\infty$.
It arises as cubic base change from the rational elliptic surface
with $s=t^3$.\\
$\MW = \Z/4\Z \times A_2[2/3]$.
Torsion generated by $(0,0)$;
minimal sections $(t^3 + \varrho t^2, \varrho^2 t^4)$ for $\varrho^3=1$
and their negatives.

Over $\Q$ arithmetic and geometry of this fibration have been studied in
detail in \cite{S-MJM}.
In particular, the connection to \#8 has been worked out over $\Q$,
and a divisor of type $\td_{20}$ as in \#18 has been identified over $\F_4$,
albeit without expressing its linear system in terms of the above
Weierstrass form.

\subsection{\#7: $R(M) = A_3E_6A_{11}$}

Model for instance $[1,0,t^4,0,0]$.\\
Singular fibers
$\ta_{11}$ at $t=0$, $\ta_3$ at $t=1$ and $\te_6$ at $\infty$.\\
$\MW=\Z/6\Z$, generated by $P=(t^2,t^2)$.
3-torsion: $4P=(0,0)$, 2-torsion: $3P=(t^4,t^6)$.

\subsubsection*{Connection with \#4}

$u =(y-x) / (t(x-t^2))$ 
extracts two divisors of type $\tilde A_9$ from $\tilde A_{11}$
and $\tilde E_6$ connected
by zero section and 6-torsion section $5P=(t^2,t^4)$ on the one
hand and by $P, 4P$ on the other hand.
The odd components of $\tilde A_3$ are not met by any section and
thus form two $A_1$'s.

There are three new sections given by fiber components as shown
in the figure plus $2P, 3P$ and the even components of $\tilde
A_3$.

\begin{figure}[ht!]
\setlength{\unitlength}{.45in}
\begin{picture}(10,5.8)(-0.7,0)
\thicklines

\put(6,5){\circle*{.1}}
\put(7,5){\circle*{.1}}
\put(7,5){\circle{.2}}
\put(8,4.5){\circle*{.1}}
\put(8.5,3.5){\circle*{.1}}
\put(8.5,2.5){\circle*{.1}}
\put(8,1.5){\circle*{.1}}
\put(7,1){\circle*{.1}}
\put(6,1){\circle*{.1}}
\put(5,4.5){\circle*{.1}}
\put(4.5,3.5){\circle*{.1}}
\put(4.5,2.5){\circle*{.1}}
\put(5,1.5){\circle*{.1}}

\put(6,5){\line(1,0){1}}
\put(7,5){\line(2,-1){1}}
\put(8,4.5){\line(1,-2){.5}}
\put(8.5,3.5){\line(0,-1){1}}
\put(8.5,2.5){\line(-1,-2){.5}}
\put(8,1.5){\line(-2,-1){1}}
\put(7,1){\line(-1,0){1}}
\put(6,1){\line(-2,1){1}}
\put(5,4.5){\line(2,1){1}}
\put(4.5,2.5){\line(0,1){1}}
\put(4.5,2.5){\circle{.2}}

\put(4.5,3.5){\line(1,2){.5}}
\put(5,1.5){\line(-1,2){.5}}

\put(3.5,3.5){\circle{.1}}
\put(3.55,3.65){$O$}

\multiput(0.5,3.5)(1,0){3}{\circle*{.1}}

\multiput(-1,5)(0.75,-0.75){2}{\circle*{.1}}
\multiput(-1,2)(0.75,0.75){2}{\circle*{.1}}






\put(-1,5){\line(1,-1){1.5}}
\put(-1,2){\line(1,1){1.5}}
\put(.5,3.5){\line(1,0){4}}
\put(-0.25,2.75){\circle{.2}}


\put(2.5,5){\circle{.1}}
\put(-1,5){\line(1,0){7}}
\put(2.55,4.7){$5P$}

\put(2,2){\circle{.1}}
\put(2,2){\line(6,-1){3}}
\put(-1,2){\line(1,0){3}}
\put(2.05,1.6){$P$}

\put(-1,2){\line(4,-1){4}}
\put(3,1){\circle{.1}}
\qbezier(3,1)(13,-1)(8,4.5)
\put(3.1,1.1){$4P$}






%

%


\thinlines

\put(-1.5,2.25){\line(0,-1){2}}
\put(-1.5,2.25){\line(1,0){9}}
\put(-1.5,0.25){\line(1,0){11.5}}
\put(10,0.25){\line(0,1){5.25}}
\put(7.5,2.25){\line(0,1){3.25}}
\put(7.5,5.5){\line(1,0){2.5}}



\put(-1.5,3){\framebox(8,2.5){}}

\end{picture}
\caption{Two $\tilde A_9$ divisors supported on $\tilde E_6, \tilde
A_{11}$ and torsion sections}
\label{Fig:7-4}
\end{figure}
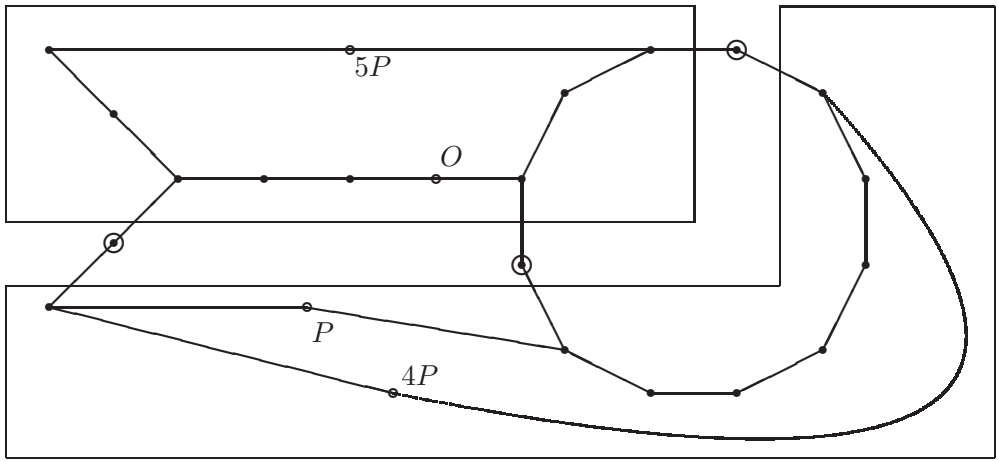

\subsubsection*{Connection with \#8}

$u = (x-t^3)/(t^4-t^3)$ extracts two $\tilde E_6$'s from $\ta_3$
and $\ta_{11}$ connected through $O$ and $3P$. The third copy of
$\tilde E_6$ comes from the root lattice $E_6$ of non-identity
components of the original $\te_6$ fiber.

\begin{figure}[ht!]
\setlength{\unitlength}{.45in}
\begin{picture}(10,4.8)(0,.7)
\thicklines

\put(6,5){\circle*{.1}}
\put(7,5){\circle*{.1}}
\put(7,5){\circle{.2}}

\put(8,4.5){\circle*{.1}}
\put(8.5,3.5){\circle*{.1}}
\put(8.5,2.5){\circle*{.1}}
\put(8,1.5){\circle*{.1}}
\put(7,1){\circle*{.1}}
\put(6,1){\circle*{.1}}
\put(5,4.5){\circle*{.1}}
\put(4.5,3.5){\circle*{.1}}
\put(4.5,2.5){\circle*{.1}}
\put(5,1.5){\circle*{.1}}

\put(6,5){\line(1,0){1}}
\put(7,5){\line(2,-1){1}}
\put(8,4.5){\line(1,-2){.5}}
\put(8.5,3.5){\line(0,-1){1}}
\put(8.5,2.5){\line(-1,-2){.5}}
\put(8,1.5){\line(-2,-1){1}}
\put(7,1){\line(-1,0){1}}
\put(6,1){\line(-2,1){1}}
\put(6,1){\circle{.2}}

\put(5,4.5){\line(2,1){1}}
\put(4.5,2.5){\line(0,1){1}}
\put(4.5,3.5){\line(1,2){.5}}
\put(5,1.5){\line(-1,2){.5}}

\put(3.5,3.5){\circle{.1}}
\put(3.55,3.65){$O$}
\put(2.5,3.5){\line(1,0){2}}

\multiput(1.5,4.5)(1,-1){2}{\line(-1,-1){1}}
\multiput(1.5,4.5)(1,-1){2}{\circle*{.1}}

\multiput(1.5,4.5)(-1,-1){2}{\line(1,-1){1}}
\multiput(0.5,3.5)(1,-1){2}{\circle*{.1}}
\multiput(1.5,2.5)(0,2){2}{\circle{.2}}

\put(0.5,3.5){\line(-3,-1){1}}
\put(8.5,2.5){\line(3,1){2}}
\put(10,3){\circle*{.1}}
\put(9.6,3.1){$3P$}
















%

%


\thinlines


\put(2.1,1.4){$u=\infty$}
\put(9.6,0.9){$u=0$}

\put(2,1.25){\framebox(4.25,4){}}

\put(6.75,0.75){\line(1,0){3.75}}
\put(6.75,0.75){\line(0,1){4}}
\put(6.75,4.75){\line(1,0){3.75}}

\put(1,0.75){\line(-1,0){1.5}}
\put(1,0.75){\line(0,1){4}}
\put(1,4.75){\line(-1,0){1.5}}

\end{picture}
\caption{Two $\tilde E_6$ divisors supported on $\tilde A_3, \tilde
A_{11}$ and 2-torsion sections}
\label{Fig:7-8}
\end{figure}
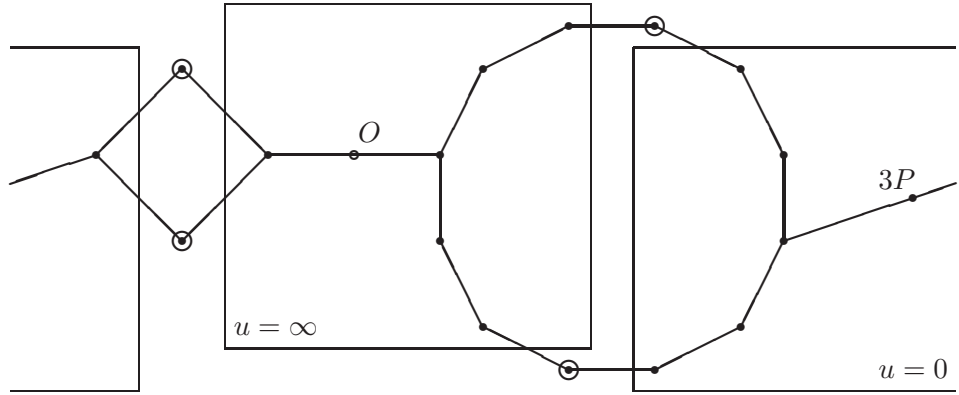


%


\subsubsection*{Connection with \#11}
$u = (y-t^2) / (t(x-t^2))$
extracts $\tilde A_{17}$ from $\te_6, \ta_{11}$ connected through
zero section and $P$.
Contrary to the connection with \#4, we choose the long way around
the $\ta_{11}$ fiber.
This leaves three $A_1$'s comprising a far component of $\tilde
E_6$ as shown in the figure and the odd components of $\tilde A_3$.
On top of the indicated fiber component, we obtain new sections
from the even components of $\tilde A_3$ and $2P, 5P$.

\begin{figure}[ht!]
\setlength{\unitlength}{.45in}
\begin{picture}(10,5.8)(-1.2,0)
\thicklines

\put(6,5){\circle*{.1}}
\put(7,5){\circle*{.1}}
\put(8,4.5){\circle*{.1}}
\put(8.5,3.5){\circle*{.1}}
\put(8.5,2.5){\circle*{.1}}
\put(8,1.5){\circle*{.1}}
\put(7,1){\circle*{.1}}
\put(6,1){\circle*{.1}}
\put(5,4.5){\circle*{.1}}
\put(4.5,3.5){\circle*{.1}}
\put(4.5,2.5){\circle*{.1}}
\put(5,1.5){\circle*{.1}}

\put(6,5){\line(1,0){1}}
\put(7,5){\line(2,-1){1}}
\put(8,4.5){\line(1,-2){.5}}
\put(8.5,3.5){\line(0,-1){1}}
\put(8.5,2.5){\line(-1,-2){.5}}
\put(8,1.5){\line(-2,-1){1}}
\put(7,1){\line(-1,0){1}}
\put(6,1){\line(-2,1){1}}
\put(5,4.5){\line(2,1){1}}
\put(4.5,2.5){\line(0,1){1}}
\put(4.5,3.5){\line(1,2){.5}}
\put(5,1.5){\line(-1,2){.5}}

\put(3.5,3.5){\circle{.1}}
\put(3.55,3.65){$O$}

\multiput(0.5,3.5)(1,0){3}{\circle*{.1}}

\multiput(-1,5)(0.75,-0.75){2}{\circle*{.1}}
\multiput(-1,2)(0.75,0.75){2}{\circle*{.1}}

\put(-1,5){\line(1,-1){1.5}}
\put(-1,2){\line(1,1){1.5}}
\put(.5,3.5){\line(1,0){4}}

\put(-.25,2.75){\circle{.2}}


\put(2.5,5){\circle{.1}}
\put(-1,5){\line(1,0){7}}
\put(2.55,4.7){$P$}



\thinlines

\put(-1.5,3){\line(0,1){2.5}}
\put(-1.5,3){\line(1,0){5}}
\put(-1.5,5.5){\line(1,0){10.5}}
\put(9,5.5){\line(0,-1){5}}
\put(3.5,3){\line(0,-1){2.5}}
\put(3.5,0.5){\line(1,0){5.5}}



\put(4.75,4.25){\framebox(0.5,0.5){}}

\end{picture}
\caption{$\tilde A_{17}$ divisor supported on $\tilde E_6, \tilde
A_{11}$ and torsion sections}
\label{Fig:7-11}
\end{figure}
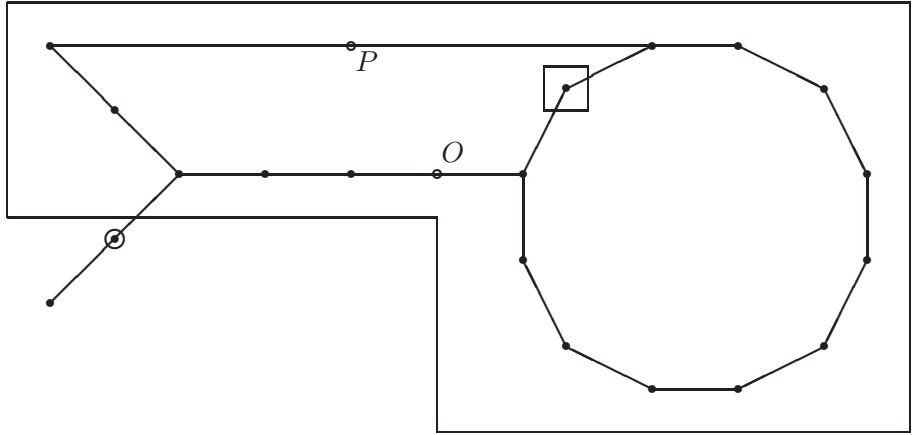




\subsubsection*{Connection with \#13}

$u = x/t^4$ extracts two $\tilde E_7$'s first from $\tilde A_{11}$
adjoined by the zero section
and secondly from $\tilde E_6$ adjoined by $2P, 4P$.
Remaining components of $\tilde A_{11}$ combine with $3P$ and $A_3$
($\ta_3$ minus identity component) to $\tilde D_6$.
Two sections given by fiber components as depicted.

\begin{figure}[ht!]
\setlength{\unitlength}{.45in}
\begin{picture}(11,7)(-0.75,0.25)
\thicklines


\put(4,3){\circle*{.1}}
\put(4.25,4){\circle*{.1}}
\put(5,4.75){\circle*{.1}}
\put(6,5){\circle*{.1}}

\put(8,3){\circle*{.1}}
\put(7.75,4){\circle*{.1}}
\put(7,4.75){\circle*{.1}}

\put(7,4.75){\circle{.2}}
\put(7,1.25){\circle{.2}}
\put(4,3){\line(1,4){0.25}}
\put(4.25,4){\line(1,1){.75}}
\put(5,4.75){\line(4,1){1}}

\put(8,3){\line(-1,4){0.25}}
\put(7.75,4){\line(-1,1){.75}}
\put(7,4.75){\line(-4,1){1}}

\put(4.25,2){\circle*{.1}}
\put(5,1.25){\circle*{.1}}
\put(6,1){\circle*{.1}}

\put(7.75,2){\circle*{.1}}
\put(7,1.25){\circle*{.1}}

\put(4,3){\line(1,-4){0.25}}
\put(4.25,2){\line(1,-1){.75}}
\put(5,1.25){\line(4,-1){1}}

\put(8,3){\line(-1,-4){0.25}}
\put(7.75,2){\line(-1,-1){.75}}
\put(7,1.25){\line(-4,-1){1}}

\put(3,3){\circle{.1}}
\put(3.05,3.15){$O$}

\multiput(0,3)(1,0){3}{\circle*{.1}}

\multiput(-1,4)(0,-2){2}{\circle*{.1}}
\multiput(-.5,3.5)(0,-1){2}{\circle*{.1}}

\put(-1,4){\line(1,-1){1}}
\put(-1,2){\line(1,1){1}}
\put(0,3){\line(1,0){4}}

\put(9,3){\circle{.1}}

\put(9.1,3.15){$3P$}
\qbezier(2,3)(3,11)(9,3)

\put(10,3){\circle*{.1}}
\put(10.5,3.5){\circle*{.1}}
\put(10.5,2.5){\circle*{.1}}

\put(8,3){\line(1,0){2}}
\put(10,3){\line(1,1){.5}}
\put(10,3){\line(1,-1){.5}}

\put(1,1){\circle{.1}}
\put(1,1){\line(-2,1){2}}
\put(1,1.2){$2P$}
\qbezier(1,1)(6,-.5)(7,1.25)

\put(1,5){\circle{.1}}
\put(1,5){\line(-2,-1){2}}
\put(1,4.6){$4P$}
\qbezier(1,5)(6,6.5)(7,4.75)







\thinlines



\put(2.9,.9){$u=\infty$}
\put(-1.1,0.9){$u=0$}

\put(-1.25,0.75){\framebox(2.75,4.5){}}

\put(2.75,.75){\framebox(3.5,4.5){}}

\put(7.5,1.5){\framebox(3.25,3){}}

\end{picture}
\caption{Two $\tilde E_7$ and $\tilde D_6$ supported on $\tilde
E_6, \tilde A_{11}, A_3$ and sections}
\label{Fig:7-13}
\end{figure}
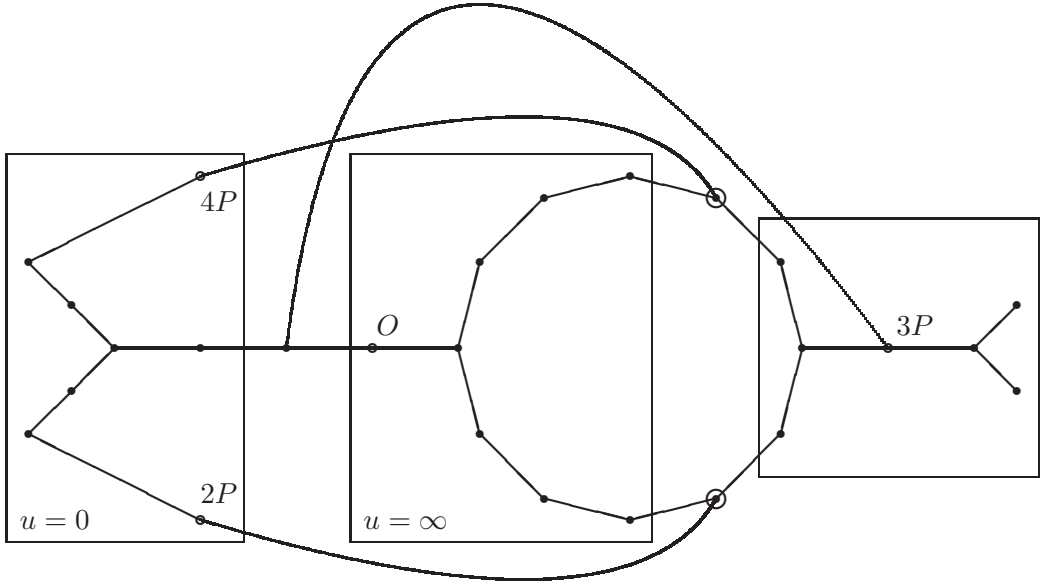




\subsubsection*{Connection with \#14}

$u = (x-1)/(t-1)^2$ extracts $\tilde D_8$ from $\tilde E_6$ and
$\tilde A_3$ connected through zero section.
$\tilde D_{12}$ given by $A_{11}$ extended by sections $P, 5P$.
Far components of $\tilde E_6$ serve as new sections.

\begin{figure}[ht!]
\setlength{\unitlength}{.45in}
\begin{picture}(11,6.5)(-0.75,0)
\thicklines


\put(4,3){\circle*{.1}}
\put(4.25,4){\circle*{.1}}
\put(5,4.75){\circle*{.1}}
\put(6,5){\circle*{.1}}

\put(8,3){\circle*{.1}}
\put(7.75,4){\circle*{.1}}
\put(7,4.75){\circle*{.1}}

%
\put(4,3){\line(1,4){0.25}}
\put(4.25,4){\line(1,1){.75}}
\put(5,4.75){\line(4,1){1}}

\put(8,3){\line(-1,4){0.25}}
\put(7.75,4){\line(-1,1){.75}}
\put(7,4.75){\line(-4,1){1}}

\put(4.25,2){\circle*{.1}}
\put(5,1.25){\circle*{.1}}
\put(6,1){\circle*{.1}}

\put(7.75,2){\circle*{.1}}
\put(7,1.25){\circle*{.1}}

\put(4,3){\line(1,-4){0.25}}
\put(4.25,2){\line(1,-1){.75}}
\put(5,1.25){\line(4,-1){1}}

\put(8,3){\line(-1,-4){0.25}}
\put(7.75,2){\line(-1,-1){.75}}
\put(7,1.25){\line(-4,-1){1}}

\put(3,3){\circle{.1}}
\put(2.95,3.2){$O$}

\multiput(1.5,4.5)(.5,-.5){3}{\circle*{.1}}

\multiput(0,4.5)(.75,0){2}{\circle*{.1}}
\put(0,4.5){\line(1,0){1.5}}
\put(0,4.5){\circle{.2}}

\multiput(1.5,6)(0,-.75){2}{\circle*{.1}}
\put(1.5,6){\line(0,-1){1.5}}
\put(1.5,6){\circle{.2}}


\put(1.5,4.5){\line(1,-1){1.5}}
\put(3,3){\line(1,0){1}}

\put(2.25,2.25){\circle*{.1}}
\put(1.5,2.25){\circle*{.1}}
\put(2.25,1.5){\circle*{.1}}

\put(2.25,2.25){\line(1,1){.75}}
\put(2.25,2.25){\line(0,-1){.75}}
\put(2.25,2.25){\line(-1,0){.75}}

\put(4.25,5.5){\circle{.1}}
\put(4.25,5.5){\line(1,-1){.75}}
\put(4.4,5.5){$P$}

\qbezier(4.25,5.5)(3,6)(1.5,6)

\put(4.25,0.5){\circle{.1}}
\put(4.25,0.5){\line(1,1){.75}}
\put(4.4,0.3){$5P$}

\qbezier(4.25,0.5)(-1,-1)(0,4.5)

%












\thinlines

\put(.6,1.4){$u=\infty$}
\put(7.3,0.15){$u=0$}

\put(.5,1.25){\framebox(2.75,4.25){}}


\put(4.125,0){\framebox(4.125,6){}}

\end{picture}
\caption{$\tilde D_8$ and $\tilde D_{12}$ supported on $\tilde
A_3, \tilde E_6, \tilde A_{11}$ and sections}
\label{Fig:7-14}
\end{figure}
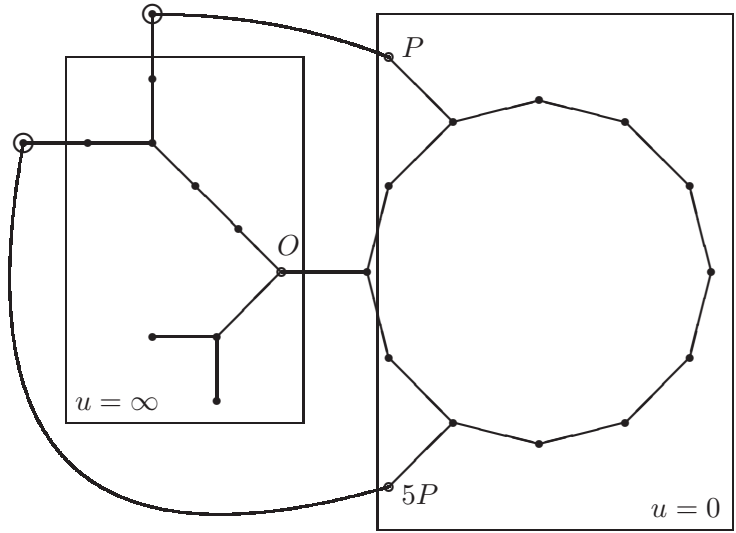

\subsection{\#8: $R(M) = E_6^3$}

Model for instance $[0, 0, t^2 (t+1)^2, 0, 0]$, as investigated
in \cite{S-MJM}.
Singular fibers at $t=0,1,\infty$.
$\MW = A_2[2/3] \times \Z/3\Z$.
Torsion generated by $(0,0)$.
Minimal sections $(\varrho t^2, t^2)$ and their negatives.



\subsection{\#9: $R(M) = D_4D_8^2$}

$[0,0,0,t^2(t^4+t^2+1),t^5(t^2+1)]$.\\
Singular fibers $\td_8$ at $t=0,\infty$,  $\td_4$ at $t=1$.\\
$\MW=(\Z/2\Z)^2$    with sections	$(t,0), (t^3,0), (t^3+t,0)$

\subsection{\#10:  $R(M) = D_5 A_{15}$}

$[t^2, 0, 0, 1, 0]$\\
Singular fibers $\td_5$ at $t=0$, $\ta_{15}$ at $\infty$.\\
$\MW=\Z/4\Z$,  generated by $P=(1,0)$ with 2-torsion at $(0,0)$.


\subsubsection*{Connection with \#6}

$u = (x+t+1) / t^2$ extracts $\tilde D_7$ from $\tilde D_5, \tilde
A_{15}$ connected through zero section.
The disjoint components of $\tilde A_{15}$ form an $A_{11}$.
New sections as depicted plus $P, 3P$.

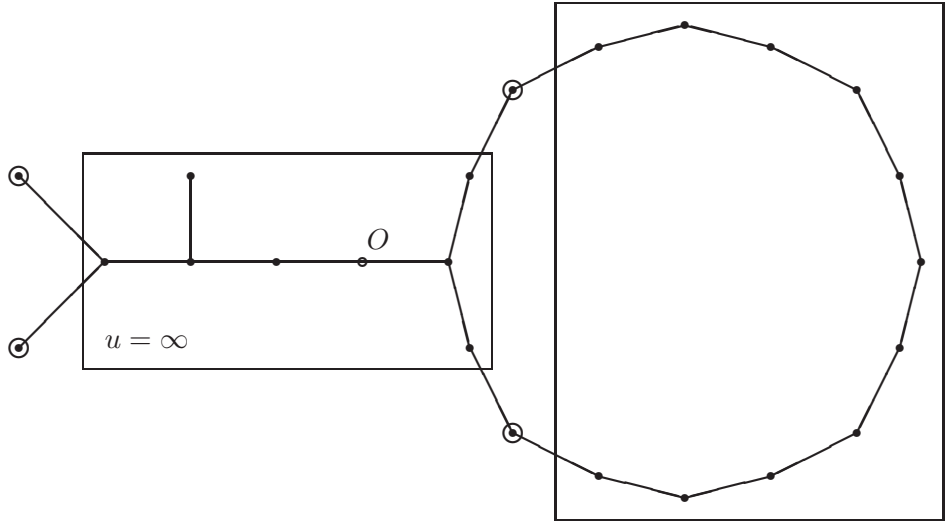
\begin{figure}[ht!]
\setlength{\unitlength}{.45in}
\begin{picture}(10,6)(-0.75,0)
\thicklines

\put(4,3){\circle*{.1}}
\put(4.25,4){\circle*{.1}}
\put(4.75,5){\circle*{.1}}
\put(5.75,5.5){\circle*{.1}}
\put(6.75,5.75){\circle*{.1}}

\put(9.5,3){\circle*{.1}}
\put(9.25,4){\circle*{.1}}
\put(8.75,5){\circle*{.1}}
\put(7.75,5.5){\circle*{.1}}

\put(4,3){\line(1,4){0.25}}
\put(4.25,4){\line(1,2){.5}}
\put(4.75,5){\line(2,1){1}}
\put(5.75,5.5){\line(4,1){1}}

\put(9.5,3){\line(-1,4){0.25}}
\put(9.25,4){\line(-1,2){.5}}
\put(8.75,5){\line(-2,1){1}}
\put(7.75,5.5){\line(-4,1){1}}

\put(4.25,2){\circle*{.1}}
\put(4.75,1){\circle*{.1}}
\put(5.75,0.5){\circle*{.1}}
\put(6.75,0.25){\circle*{.1}}

\put(9.25,2){\circle*{.1}}
\put(8.75,1){\circle*{.1}}
\put(7.75,0.5){\circle*{.1}}

\put(4,3){\line(1,-4){0.25}}
\put(4.25,2){\line(1,-2){.5}}
\put(4.75,1){\line(2,-1){1}}
\put(5.75,0.5){\line(4,-1){1}}

\put(9.5,3){\line(-1,-4){0.25}}
\put(9.25,2){\line(-1,-2){.5}}
\put(8.75,1){\line(-2,-1){1}}
\put(7.75,0.5){\line(-4,-1){1}}


\put(3,3){\circle{.1}}
\put(3.05,3.15){$O$}

\multiput(0,3)(1,0){3}{\circle*{.1}}

\multiput(-1,4)(0,-2){2}{\circle*{.1}}
\put(1,4){\circle*{.1}}

\put(-1,4){\line(1,-1){1}}
\put(-1,2){\line(1,1){1}}
\put(0,3){\line(1,0){4}}
\put(1,4){\line(0,-1){1}}


\multiput(4.75,5)(0,-4){2}{\circle{.2}}
\multiput(-1,4)(0,-2){2}{\circle{.2}}




\thinlines



\put(0,2){$u=\infty$}

\put(-0.25,1.75){\framebox(4.75,2.5){}}

\put(5.25,0){\framebox(4.5,6){}}

\end{picture}
\caption{$\tilde D_7$ and $A_{11}$ supported on $\tilde D_5, \tilde
A_{15}$ and zero section}
\label{Fig:10-6}
\end{figure}


\subsection{\#11:  $R(M) = A_1^3 A_{17}$}

$[t^2, 0, 1, 0, 0]$\\
$\ta_1$'s at third roots of unity, $\ta_{17}$ at $\infty$.\\
$\MW=\Z/6\Z$, generated by $(t,1)$.
This fibration appears in \cite[App.~2]{OS} for the peculiar fact
that it admits the 2-torsion section $(1/t^2,1/t^3)$
which is not disjoint from the zero section
(this is impossible if order and characteristic are coprime).




%



\subsection{\#12: $R(M) = A_1^3E_7D_{10}$}


quasi-elliptic $[0,0,0,t^2(t^3+1),0]$.\\
Reducible fibers: $\td_{10}$ at $t=0$, $\te_7$ at $\infty$ and
$\ta_1$'s at third roots of unity.\\
$\MW=(\Z/2\Z)^2$ with sections $P=(0,0), Q=(t,t^3), (t^4+t,t^6+t^3).$







%


\subsubsection*{Connection with \#15}
$u = x/t^2$ extracts $\tilde D_{16}$ from $\tilde E_7$ and $D_{10}$
connected through zero section.
Far component of $\tilde E_7$ combines with section $P$ and non-identity
components of $\tilde A_1$'s to form $\tilde D_4$.

\begin{figure}[ht!]
\setlength{\unitlength}{.45in}
\begin{picture}(12,4.2)(0,0.5)
\thicklines

\multiput(2,3)(1,0){9}{\circle*{.1}}
\put(2,3){\line(1,0){8}}

\multiput(2,1)(1,0){7}{\circle*{.1}}
\put(2,1){\line(1,0){6}}

\put(3,1){\circle{.2}}

\put(3,4){\circle*{.1}}
\put(3,3){\line(0,1){1}}
\put(9,4){\circle*{.1}}
\put(9,3){\line(0,1){1}}

\put(5,2){\circle*{0.1}}
\put(5,1){\line(0,1){1}}


%


\put(9,2){\circle{.1}}
\put(10,3){\line(-1,-1){2}}
\put(8.6,2){$O$}



\thinlines

\put(1.5,3){\line(0,1){1.5}}
\put(1.5,3){\line(1,-1){2.5}}
\put(1.5,4.5){\line(1,0){6}}
\put(10.5,3){\line(0,-1){2.5}}
\put(10.5,0.5){\line(-1,0){6.5}}
\put(10.5,3){\line(-2,1){3}}





\end{picture}
\caption{$\tilde D_{16}$ divisor supported on $\tilde E_7, \tilde
D_{10}$ and zero section}
\label{Fig:12-15}
\end{figure}
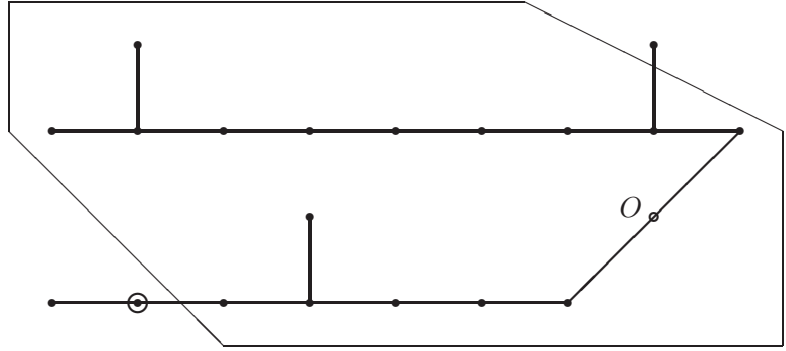

%




\subsection{\#13:  $R(M) = D_6 E_7^2$}

Quasielliptic $[0, 0, 0, t^5+t^3, 0]$\\
Reducible singular fibers $D_6, E_7, E_7$ at $t=1,0,\infty$. \\
$\MW=\Z/2\Z$ generated by $P=(0,0)$.


\subsection{\#14:  $R(M) = D_8 D_{12}$}

Quasielliptic $[0, t, 0, t^6, 0]$.\\
Reducible fibers $D_{12}$ at $t=0$ and $D_8$ at $t=\infty$. \\
$\MW=\Z/2\Z$ generated by $P=(0,0)$.



\subsubsection*{Connection with \#16}

$u=x/t^4$ extracts $\tilde E_8$ from $\tilde D_{12}$ adjoined the
zero section.
$D_8$ then combines with $P$ and remaining components of $\tilde
D_{12}$ to form a new copy of $\tilde D_{12}$.

\begin{figure}[ht!]
\setlength{\unitlength}{.45in}
\begin{picture}(12,4.2)(-0.85,0.5)
\thicklines

\multiput(0,3)(1,0){11}{\circle*{.1}}
\put(0,3){\line(1,0){10}}

\multiput(0,1)(1,0){7}{\circle*{.1}}
\put(0,1){\line(1,0){6}}

\put(7,3){\circle{.2}}

\put(1,4){\circle*{.1}}
\put(1,3){\line(0,1){1}}
\put(9,4){\circle*{.1}}
\put(9,3){\line(0,1){1}}

\put(6,1.5){\circle*{0.1}}
\put(5,1){\line(2,1){1}}

\put(2,1.5){\circle*{.1}}
\put(1,1){\line(2,1){1}}


%


\put(0,2){\circle{.1}}
\put(0,1){\line(0,1){2}}
\put(0.1,2.1){$O$}

\put(5,2){\circle{.1}}
\put(10,3){\line(-5,-1){5}}
\put(5,2){\line(-6,-1){3}}
\put(5.05,1.6){$P$}


\thinlines

\put(-0.5,4.5){\line(0,-1){2.916}}
\put(-0.5,4.5){\line(1,0){7}}
\put(6.5,4.5){\line(0,-1){1.7}}

\put(6.5,2.75){\line(-6,-1){7}}

\put(7.5,4.5){\line(0,-1){1.75}}
\put(7.5,2.75){\line(-6,-1){7}}
\put(7.5,4.5){\line(1,0){3}}
\put(10.5,4.5){\line(0,-1){4}}
\put(10.5,0.5){\line(-1,0){10}}
\put(.5,.5){\line(0,1){1.083}}


\put(-0.35,4.15){$u=\infty$}
\put(9.5,0.7){$u=0$}


\end{picture}
\caption{$\tilde E_8$ and $\tilde D_{12}$ divisors supported on
$\tilde D_8, \tilde D_{12}$ and sections}
\label{Fig:14-16}
\end{figure}
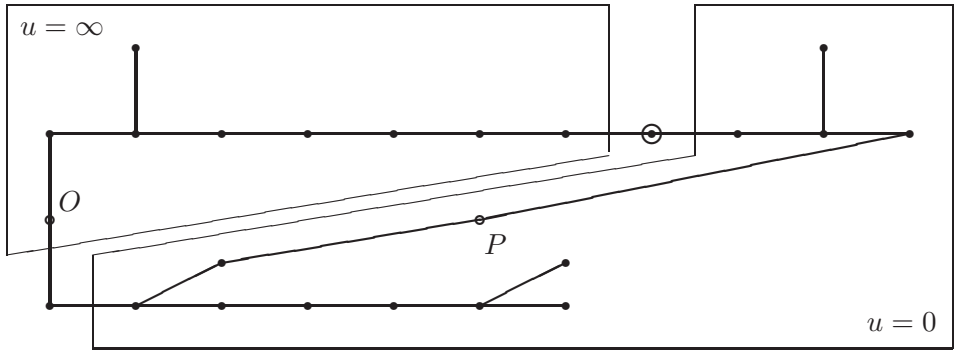


\subsection{\#15: $R(M) = D_4 D_{16}$}

Quasi-elliptic $[0,t^3,0,0,t^3]$.\\
Reducible singular fibers $\td_4$ at $t=0$, $\td_{16}$ at $\infty$.\\
$\MW=\Z/2\Z$ with section $(1,1)$.

%








\subsection{\#16: $R(M) = E_8D_{12}$}
quasi-elliptic $[0,t^3,0,0,t^5]$.\\
Reducible singular fibers $\te_8$ at $t=0, \td_{12}$ at $\infty$.



%


%







\subsubsection*{Connection with \#18}

$u=(x+t^4)/t^3$ extracts $\tilde D_{20}$ from $\tilde E_8$ and
$\tilde D_{12}$ connected by zero section.

\begin{figure}[ht!]
\setlength{\unitlength}{.45in}
\begin{picture}(11,5.4)(-.4,-0.5)
\thicklines




\multiput(0,3)(1,0){11}{\circle*{.1}}
\put(0,3){\line(1,0){10}}

\multiput(1,1)(1,0){8}{\circle*{.1}}
\put(1,1){\line(1,0){7}}

\put(1,1){\circle{.2}}

\put(1,4){\circle*{.1}}
\put(1,3){\line(0,1){1}}
\put(9,4){\circle*{.1}}
\put(9,3){\line(0,1){1}}

\put(3,0){\circle*{0.1}}
\put(3,0){\line(0,1){1}}


%


\put(9,2){\circle{.1}}
\put(10,3){\line(-1,-1){2}}
\put(8.6,2){$O$}



\thinlines

\put(-0.5,3){\line(0,1){1.5}}
\put(-0.5,3){\line(1,-1){3.5}}
\put(-0.5,4.5){\line(1,0){8}}
\put(10.5,3){\line(0,-1){3.5}}
\put(10.5,-0.5){\line(-1,0){7.5}}
\put(10.5,3){\line(-2,1){3}}





\end{picture}
\caption{$\tilde D_{20}$ divisor supported on $\tilde E_8, \tilde
D_{12}$ and zero section}
\label{Fig:16-18}
\end{figure}
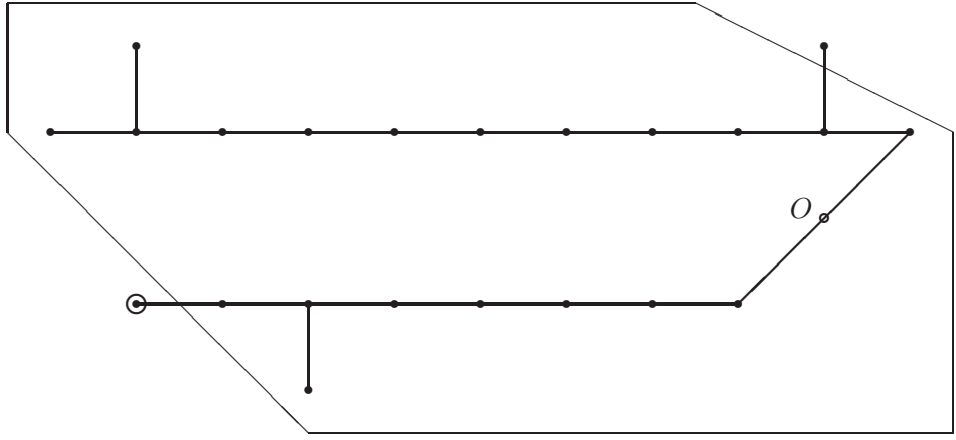


\subsection{\#17:  $R(M) = D_4 E_8^2$}

quasi-elliptic: $[0,0,0,0,t^5+t^7]$\\
Reducible fibers: $\tilde D_4$ at $t=1$, $\tilde E_8$ at $0, \infty$.
This fibration also features in \cite{Schroeer}, for instance.




%

\subsubsection*{Connection with \#15}

$u = x/t^2$ extracts $\tilde D_{16}$ from the two $\tilde E_8$'s
connected by the zero section.
Far components of $\tilde E_8$ serve as zero and 2-torsion section.
$D_4$ is preserved; the additional component to form a new $\tilde
D_4$ consists in the curve
\[
C=\{x=0, y^2=t^5(t+1)^2\}.
\]
which only meets the double component of $\tilde D_4$ and the far
components of the two $\tilde E_8$'s.

\begin{figure}[ht!]
\setlength{\unitlength}{.45in}
\begin{picture}(7,5.2)(-1.75,-.5)
\thicklines

\multiput(0,3)(1,0){8}{\circle*{.1}}
\put(0,3){\line(1,0){7}}

\multiput(0,1)(1,0){8}{\circle*{.1}}
\put(0,1){\line(1,0){7}}

\put(7,1){\circle{.2}}
\put(7,3){\circle{.2}}

\put(5,4){\circle*{.1}}
\put(5,3){\line(0,1){1}}
\put(5,0){\circle*{.1}}
\put(5,0){\line(0,1){1}}




%


\put(0,2){\circle{.1}}
\put(0,1){\line(0,1){2}}
\put(0.1,2.1){$O$}

\multiput(-3,2)(1,0){3}{\circle*{.1}}
\multiput(-2,3)(0,-2){2}{\circle*{.1}}
\put(0,2){\line(-1,0){3}}
\put(-2,1){\line(0,1){2}}

\thinlines

\put(-0.4,-.35){$u=\infty$}
\put(-3.4,0.65){$u=0$}

\put(-0.5,-0.5){\framebox(7,5){}}
\put(-3.5,0.5){\framebox(2,3){}}

\end{picture}
\caption{$\tilde D_{16}$ divisor supported on two $\tilde E_8$'s
and zero section}
\label{Fig:17-15}
\end{figure}
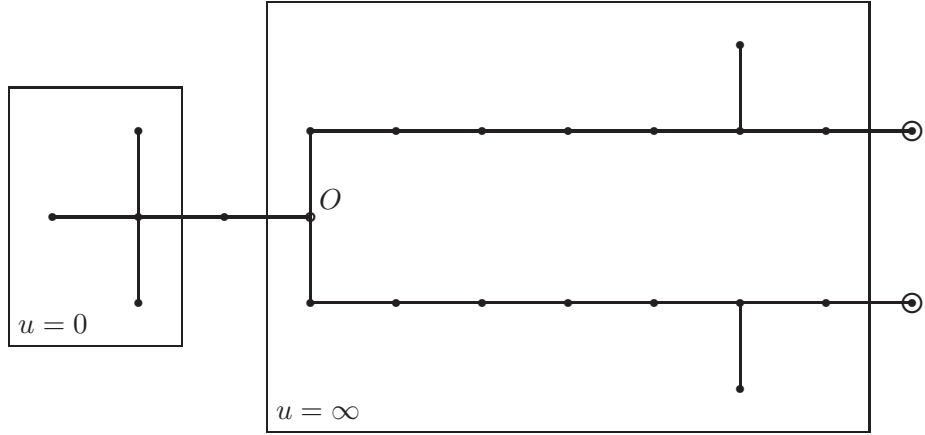


\subsection{\#18:  $R(M) = D_{20}$}

quasi-elliptic, e.g.~$[0,t^3,0,0,t]$ with $\tilde D_{20}$ at $\infty$.







%

\section{Uniqueness of the genus~1 fibrations}
\label{s:unique}

In the previous section, we have proved that the supersingular K3 surface
$X$\/ admits each  genus~$1$ fibration from Table \ref{Tab:fibr}.
The proof of Theorem \ref{thm} will thus be completed by showing
the uniqueness of each fibration.
Here it could be possible to argue with the automorphism group
of $X$ or  to pursue other lattice theoretic ideas.
We decided to follow a different approach following \cite{RS2}
that illustrates how quasi-elliptic fibrations can be used
to work out models and moduli of supersingular K3 surfaces.
Namely the uniqueness problem is stated purely in terms of genus
one fibrations:

\begin{Proposition}
\label{prop}
Let $k$ be an algebraically closed field of characteristic two.
For each genus~1 fibration from Table \ref{Tab:fibr},
there is exactly one model over $k$ up to isomorphism.
\end{Proposition}

\begin{Remark}
In comparison, on a general Kummer surface of product type
the configuration of singular fibers does usually not determine
a unique elliptic fibration by \cite{Oguiso}.
This is visible from the $2$-torsion points,
see the equations in \cite{KS}.
\end{Remark}

\begin{proof}[Proof of Proposition \ref{prop} for elliptic fibrations]
Suppose $S\to\PP^1$ is an elliptic fibration from Table
\ref{Tab:fibr}.
If the fibration is extremal, then it is a purely inseparable base
change of an extremal rational elliptic surface by \cite{Ito2}.
The uniqueness thus follows from the corresponding statement for
rational elliptic surfaces (cf.~\cite{Ito2}).
For \#11, an alternative proof can be found in \cite{SS2}.

For the remaining three elliptic fibrations, we can still argue
with extremal elliptic surfaces because there is either $3$- or
$4$-torsion in $\MW(S)$.
This implies that they arise from some universal elliptic
curves by base change.
For $3$-torsion and $j$-invariant zero (\#8), this universal elliptic
curve is
\[
y^2 + sy = x^3.
\]
Locating the singular fibers of type $\tilde E_6$ at $0, 1$ and
$\infty$, we deduce that the base change can only be $t\mapsto
s=t^2(t-1)^2$.
For $4$-torsion, we are dealing with the universal elliptic
curve
\begin{eqnarray}
\label{eq:4}
y^2 + xy + sy = x^3 + sx^2.
\end{eqnarray}
In any characteristic other than two, this has three singular fibers:
 type $I_4$ at $0$, $I_1$ at $s=1/16$  and $I_1^*$ at $\infty$.
In characteristic two, however, the latter two are merged, but
the fiber type $I_1^*$ stays the same with wild ramification of
index one.
That is, there are only two singular fibers, and each is reducible.
Since fibration \#6 has only two reducible fibers as well, it arises
from \eqref{eq:4} through a cyclic base change, i.e.~via $t\mapsto
s=t^3$.
Similarly, we also deduce that \#3 has no irreducible singular
fibers.
Locating the singular fibers at $0, 1$ and $\infty$, the fibration
thus comes from the base change
\[
t\mapsto s=t^2(t+1)^2.
\]
In particular, the elliptic fibration is unique, and we obtain
the model for \#3 in \eqref{eq3}.
\end{proof}

In order to complete the proof of Proposition \ref{prop},
we need a few more general facts about  quasi-elliptic fibrations.
A good general reference would be the last chapter of \cite{CD}.
We have already mentioned that an elliptic curve given by a $5$-tuple
$[a_1,a_2,a_3,a_4,a_6]$ is quasi-elliptic in characteristic two
if and only if $a_1\equiv a_3\equiv 0$.
Completing the cube, we thus obtain the "traditional" Weierstrass
form
\begin{eqnarray}
\label{eq:qe}
S:\;\;\; y^2 = x^3 + a_4 x + a_6.
\end{eqnarray}
Contrary to the usual situation, however, this equation still admits
the following automorphisms:
\[
x\mapsto x+\alpha^2, \;\;\; y\mapsto y+\alpha x+\beta
\]
in addition to rescaling $x$ and $y$ by a second resp.~third power.
Hence $a_4$ and $a_6$ are unique up to the according scaling and
adding fourth powers resp.~squares.
Quasi-elliptic fibrations  admit a discriminant that detects
the reducible singular fibers:
\[
\Delta = a_4 (a_4')^2 + (a_6')^2.
\]
Here the prime indicates the formal derivative with respect to
the parameter of the base curve $\PP^1$.
As a general rule, the order of vanishing of $\Delta$ equals the
rank of the Dynkin diagram associated to (the non-identity components
of) the  reducible singular fiber.
It suffices to distinguish two cases to normalize \eqref{eq:qe}:
\begin{enumerate}[(i)]
\item
If $\Delta$ is a square, then so is $a_4$. Thus we can set $a_6=t\sqrt\Delta$
and $a_4=\alpha^2$ where $\alpha$ does not contain any summand
with even exponent.
\item
If there is a fiber of type $III$ or $III^*$,
then $a_6\equiv 0$, and $a_4$ exactly encodes the singular fibers.
\end{enumerate}
We shall now prove the uniqueness for a few quasi-elliptic fibrations
from Table \ref{Tab:fibr}.
We choose some cases that illustrate the overall ideas.
All other fibrations can be treated along the same lines.

\begin{proof}[Proof of Proposition \ref{prop} for \#13]
Due to the singular fibers of type $\te_7$, we are in case (ii)
above, i.e.~$a_6=0$.
Then fiber types $\td_n$ and $\te_7$ require exact vanishing order
$2$ resp.~$3$ of $a_4$.
By M\"obius transformation, we can thus normalize \eqref{eq:qe}
uniquely as
\[
S:\;\; y^2 = x^3 + t^3 (t+1)^2 x.
\]
The two-torsion section $(0,0)$ implies that $\sigma=1$ as required.
\end{proof}

\begin{proof}[Proof of Proposition \ref{prop} for \#9 and \#17]
We locate the singular fiber of type $\td_4$ at $t=1$ and the other
two reducible fibers at $0$ and $\infty$.
Then $\Delta=t^8(t-1)^4$.
The above considerations reduce the Weierstrass form \eqref{eq:qe}
to
\[
S:\;\; y^2 = x^3 + (ut+vt^3)^2 x + t^7+t^5.
\]
Here the special fiber at $t=0$ has type $\te_8$ if $u=0$ and $\td_8$
otherwise;
the analogous statement holds at $t=\infty$.
We distinguish three cases.
First, if $\boldsymbol{u=v=0}$, then we derive \#17 in a unique
way.
Secondly, if $\boldsymbol{uv=0}$ without both vanishing,
then one fiber has type $\te_8$ and the other $\td_8$.
Note that such a surface has $\NS(S) = U \oplus D_4 \oplus D_8 \oplus E_8$
and thus Artin invariant $\sigma=2$, since the fiber type $\te_8$
on a quasi-elliptic surface does not accomodate $2$-torsion sections.
In other words, we derive a one-dimensional family of supersingular
K3 surfaces such that each member except for \#17 has Artin invariant
$\sigma=2$.

Finally we consider the case $\boldsymbol{uv\neq 0}$.
This yields a two-dimensional family of supersingular K3 surfaces,
such that the general member has $\NS(S)=U+D_4+2D_8$ and Artin
invariant $\sigma=3$.
Here the Artin invariant drops after either specializing to the
previous family or imposing some two-torsion section.
The fibration \#9 requires three non-trivial two-torsion sections.
Their intersection behavior with the reducible fibers can be predicted
from the height pairing as follows:

\begin{table}[ht!]
\begin{tabular}{c|ccc}
fiber & $\td_4$ & $\td_8$ & $\td_8$\\
\hline
fiber & id & far & far\\
comp & non-id & near & far\\
met & non-id & far & near
\end{tabular}
\end{table}

We first investigate a two-torsion section $P=(X,Y)$ that fits
into the first row.
Here $X$ and $Y$ are polynomial in $t$ of degree at most $4$ resp.~$6$.
At $t=0$, it is immediate that $t|X, t^2|Y$.
This corresponds to blowing up the surface
once at the point $(x,y,t)=(0,0,0)$ and then along the exceptional
divisor.
In the affine chart $x=tx', y=t^2y''$ this yields
\begin{eqnarray}
\label{eq:9-1}
S:\;\; ty''^2 = x'^3 +(u+vt)^2 x' + t^4+t^2.
\end{eqnarray}
Here the near simple component of the $\td_8$ fiber is given by
$t=x'=0$.
The section has to follow the double component $\{t=0, x'=u\}$
through the resolution,
so $X=t(u+t\hdots)$.
Successively this yields $t^3|Y$ and $X=t(u+t/\sqrt{u}+t^2\hdots)$.
By symmetry, the same argument applies to the fiber at $\infty$.
We deduce $\deg(Y)\leq 3$ and $X=t^3v+t^2/\sqrt{v}+\hdots$.
Combining the information from $t=0$ and $t=\infty$, we deduce
$u=v$
and find a unique section
$P=t(u+t/\sqrt{u}+ut^2), u^{3/2}t^3)$.
Again we have thus found a family of supersingular K3 surfaces
with Artin invariant $\sigma\leq 2$.

We continue by imposing a torsion section $Q=(\mathcal X, \mathcal
Y)$ of the second kind, say meeting the fiber at $\infty$ at a
far component.
As before, this implies $\deg(\mathcal Y)\leq 3$ and $\mathcal
X=t^3u+t^2/\sqrt{u}+\hdots$.
By \eqref{eq:9-1}, the near component of the fiber at $t=0$ is
met if and only if $t^2|\mathcal X, \mathcal Y$, so
$\mathcal X=t^3u+t^2/\sqrt{u}$.
Finally the intersection of a non-identity component at $t=1$ requires
$(t+1)|\mathcal X, \mathcal Y$.
Hence $u=1/\sqrt{u}$, i.e.~$u^3=1$.
The three possible choices are identified by scaling $x$ by third
roots of unity.
Hence we can assume $u=1$ and find the section $Q=(t^2(t+1), t^2(t+1))$.
This shows that the quasi-elliptic fibrations \#9 and \#17 are
unique.
\end{proof}

For all other quasi-elliptic fibrations from Table \ref{Tab:fibr},
uniqueness can be proven along similar lines.
The cases with five reducible fibers which at first sight might
look most complicated are greatly simplified by the following easy observation:
Any genus~$1$ fibration from Table \ref{Tab:fibr} has Artin invariant
$\sigma=1$; thus it gives a model of our supersingular K3 surface $X$.
Now $X$ has a model with all $\NS(X)$ defined over $\F_4$.
By the argumentation in Section \ref{s:g=1}, it follows that any
genus~$1$ fibration on $X$ admits such a model, too.
For the genus~$1$ fibrations with five reducible fibers, this identifies
the locus of reducible fibers on the base curve as $\PP^1(\F_4)$
which essentially fixes the Weierstrass form \eqref{eq:qe}.
Then it remains to check for precise fiber types and for fiber
components to be defined over $\F_4$.

For instance, for \#2 this means that we can work with a Weierstrass
form
\[
S:\;\; y^2 = x^3 + \alpha t^2 x + (t^3+1)^3\;\;\; (\alpha\in\F_4).
\]
Here the components of the fiber at $t=1$ are encoded in the roots
of the polynomial $T^3+\alpha T + 1$.
It is easily checked that this polynomial splits over $\F_4$
if and only if $\alpha=0$.
We derive the model for \#2 in \ref{ss:2} with \MoW\ group as specified.
The details for the remaining cases are left to the reader.

\section{Points and lines in $\PP^2(\F_4)$}
\label{s:config}

Consider the elliptic fibration \#1
with $R(L)=A_5^4$ and $\MW\cong\Z/3\Z \times \Z/6\Z$.
There are 42 obvious $(-2)$ curves 
formed by the 24 components of the singular fibers
and the 18 torsion sections.
It is easily verified that the configuration of these 42 rational curves
is the incidence graph of the 21 points and 21 lines of $\PP^2(\F_4)$ (cf.~\cite{DK}, \cite{KK}).
This gives another way to see the large finite automorphism group $\mbox{PGL}_3(\F_4)\times\Z/2\Z$
acting on $X$.
We remark that the 42 roots of $\NS(X)$ under consideration 
are known as the first Vinberg batch of roots for $\mbox{I}_{1,21}$
(which contains $\NS(X)$ as even sublattice, see \cite[p.~551]{CS}).
Note also that fiber components and sections 
over $\F_2$ induce the incidence graph of $\PP^2(\F_2)$,
so our identification is compatible with the Galois action.

For each of the other 17 fibrations in our list,
most or all of the $(-2)$ curves from $R(L)$
and torsion sections can already be seen in the $\PP^2(\F_4)$ picture.
For example, for the quasi-elliptic fibration \#2 with $R(L)=D_4^6$ and $\MW=(\Z/2\Z)^4$,
fiber components and sections give 41 rational curves
which correspond to all but one of the 42 vertices of the incidence graph.
For a few other cases, see the discussion below.

From our classification of genus 1 fibrations on $X$
we can extract information about specific subgraphs of the incidence graph:

\begin{Theorem}
\label{thm:inc}
The incidence graph of points and lines in $\PP^2(\F_4)$ 
does not contain any cycle of length $14$ or $2n$ with $n\geq 10$.
\end{Theorem}

\begin{proof}
If there were such a cycle,
then we would find a corresponding effective divisor on $X$
via the elliptic fibration \#1.
As explained in Section \ref{s:g=1},
this divisor would induce an elliptic fibration on $X$
with the cycle as singular fiber of type $I_{2n}$
(jacobian by Theorem \ref{thm:jac}).
Then the classification of genus 1 fibrations on $X$ leads to the desired contradiction. 
\end{proof}

\begin{Remark}
Alternatively one can infer $n<11$ from the Shioda-Tate formula and $n\neq 10$ from \cite{S-max},
but we are not aware of an easy argument ruling out $n=7$.
\end{Remark}

\begin{Proposition}
Let $n\in\N$.
Assume that there are $n$ points $P_i\in \PP^2(\F_4)\; (i\in\Z/n\Z)$ such that $P_i, P_{i+1}, P_j$  are never collinear for distinct $i, i+1, j$. 
Then $n\in\{3,4,5,6,8,9\}$.
Conversely for each such $n$, there is a $2n$-cycle in $\PP^2(\F_4)$.
\end{Proposition}

\begin{proof}
All other cases are ruled out by Theorem \ref{thm:inc},
so the first statement of the proposition follows.
As for the existence part,
all $2n$-cycles for $n<9$ can easily be realized in the affine plane $\A(\F_4)$
by way of horizontal and vertical lines and the diagonal, say.
As for the $18$-cycle, one can connect, for instance,
the affine points
$(0,0), (\varrho^2,0), (\varrho,1),(\varrho^2,1),(\varrho,\varrho),(\varrho^2,\varrho),(\varrho,\varrho^2),(1,\varrho^2)$ and the infinite point $[0,1,0]$.
\end{proof}

We can be even more specific by analyzing the roots perpendicular to the given $2n$-cycle
(thus forming fiber components of the induced elliptic fibration),
and the points and lines giving rise to sections.
In the counts, $a+b$ indicates the partition between points and lines in $\PP^2(\F_4)$.

\subsection{$\mathbf{\ta_5}$} There are $9+9$ disjoint roots, forming another three $\ta_5$ hexagons,
plus $9+9$ sections (roots that meet exactly one of the $\ta_5$ vertices) comprising the full $\MW$ group.
Of course, this was expected since we  started our current investigation exactly with this fibration.

\subsection{$\mathbf{\ta_7}$} 
\label{ss:a7}
$7+7$ disjoint roots, forming the remaining $\ta_7$ and $\td_5$ fibers of \#3,
and $8+8$ sections. 
Here $\MW$ has rank 1, so the sections can only comprise part of it.

\subsection{$\mathbf{\ta_9}$}
$6+6$ disjoint roots, forming the other $\ta_9$ of \#4 and two isolated  $A_1$'s; 
$5+5$ sections, accounting for
the full MW group.

\subsection{$\mathbf{\ta_{11}}$}
\label{ss:a11}
There are two possibilities.  In one case, the vertices of
the same parity on both the hexagon and its dual are always collinear.
Then there are $4+4$ disjoint roots, forming a $\td_7$ system, so we have
the case of \#6 with MW rank 2.
There are $6+6$ sections.  In the other case, either the hexagon
or its dual is a "hyperoval", with no three points collinear
(and the other has vertices of the same parity collinear).
Here there are $6+4$ disjoint roots, forming  $\te_6$ and $A_3$ of \#7.
There are $6+0$ sections, accounting for the full MW group.
(The 0 was expected because no line meets a hyperoval in exactly
one point).

\subsection{$\mathbf{\ta_{15}}$}
\label{ss:a15}
Here if we look at points of the same parity on the octagon
and its dual, three of the resulting four sets of 4 points are collinear
and the last is in general linear position.  
There are $2+3$ disjoint
roots, forming a $D_5$ root system, consistent with the case \#10.
There are $4+0$ sections (none for the octagon with
two 4-point lines), accounting for the full MW group.

\subsection{$\mathbf{\ta_{17}}$}
Just 1+1 disjoint roots, so we see only part of the $A_1^3$ .
configuration of \#11. 
(Happily the disjoining roots are also disjoint from
each other as they must be to be part of $A_1^3$.)  
There are $3+3$ sections,
again fully accounting for the MW group.

\subsection*{$\mathbf{\td_n}$ configurations}
Along similar lines, we can study other configurations
 in the incidence graph of $\PP^2(\F_4)$.
The $\td_{2n}$ series is much like  $\ta_{2n-1}$:
 instead of a polygon,
we have a path whose first and last line contain three points each
rather than two -- or dually where the first and last vertices
have two terminal lines each instead of one.  
Here the lattices in our classification 
let us see everything up to $D_{20}$ except $D_{14}$ and $D_{18}$.
Thus $\td_{14}$ and $\td_{18}$ are impossible.
We will rule out $\td_{20}$ separately below. 
Conversely, for all other $\td_{2n}, 2\leq n\leq 8$, 
the existence is easily derived from our analysis of $\ta_{2n-1}$ configurations
extended by sections.

\begin{Example}
 $\td_{16}$ is obtained from $\ta_{15}$ by attaching two sections (aka points in \ref{ss:a15})
that are not opposite while omitting the middle $(-2)$ curve (aka line) of the shorter path connecting them in the extended $\ta_{15}$ graph.
\end{Example}

We shall now disprove the existence of a configuration of type $\td_{20}$ in $\PP^2(\F_4)$.
The configuration is sketched in the following figure:

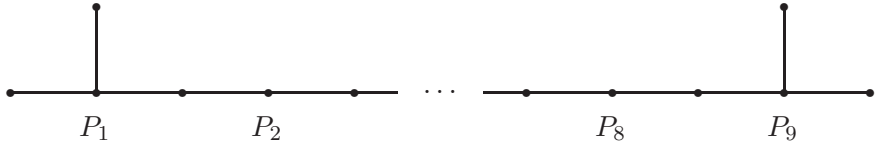
\begin{figure}[ht!]
\setlength{\unitlength}{.45in}
\begin{picture}(11,2)(-.4,.5)
\thicklines




\multiput(0,1)(1,0){5}{\circle*{.1}}
\put(0,1){\line(1,0){4.5}}

\multiput(6,1)(1,0){5}{\circle*{.1}}
\put(5.5,1){\line(1,0){4.5}}


\put(4.8,1){$\hdots$}

\put(1,2){\circle*{.1}}
\put(1,1){\line(0,1){1}}
\put(9,2){\circle*{.1}}
\put(9,1){\line(0,1){1}}

\put(.8,.5){$P_1$}
\put(2.8,.5){$P_2$}
\put(6.8,.5){$P_8$}
\put(8.8,.5){$P_9$}



%

%



%





\end{picture}
\caption{$\tilde D_{20}$ configuration in $\PP^2(\F_4)$}
\label{Fig:20}
\end{figure}

The configuration includes 3 lines through $P_1$, 
so there are 2 others which we label $\ell_1, \ell_2$.
In fact these 2 lines have to contain all points $P_3,\hdots, P_9$
which are off the 3 lines though $P_1$ from the figure, but neither contains $P_2$.
We infer that the odd-indexed points $P_3,
\hdots,P_9$ sit on $\ell_1$ and the even-indexed points $P_4,\hdots,P_8$ on $\ell_2$.
The same argument applies to $P_9$ and leads to a line $\ell_3$ containing the even-indexed points $P_2,\hdots,P_6$.
But then clearly $\ell_2=\ell_3$ containing both $P_2$ and $P_8$.
This contradicts the choice of configuration which is thus impossible on $\PP^2(\F_4)$.

Similarly for $\td_{2n-1}$ we have a path with an extra point on one side and an
extra line on the other.
From our classification we deduce that this is not possible past $\td_7$
while we have already seen $\td_5$ and $\td_7$ in \ref{ss:a7} and \ref{ss:a11}.



\section{Reduction from characteristic zero}
\label{s:red}

The classification of elliptic fibrations on $X$ enables us to
determine all elliptic K3 surfaces in characteristic zero with
good reduction at (a prime above) $2$ yielding $X$.

Let us explain why we consider this an interesting question.
The main reason is that we have plenty of possible candidates at
hand.
For instance, we could work with singular K3 surfaces (attaining
the maximal Picard number $\rho=20$ over $\C$).
Singular K3 surfaces always come with natural elliptic fibrations
from the so-called Shioda-Inose structure.
Namely there is Inose's pencil with two $II^*$ fibers and (in general)
$\MW$-rank two (cf.~\cite{Sandwich}).
But those special fibers have wild ramification in characteristic
$2$ and $3$ by \cite{SS2}, so there has to be some kind of degeneration.
In fact, one can show that for any singular K3 surface the Inose
pencil degenerates modulo (any prime above) $2$ to the quasi-elliptic
fibration \#17 (so that the reduction is not smooth due to the
$\td_4$ fiber on the reduction).
A similar pattern holds in general:

\begin{Proposition}
\label{prop:0}
Let $k$ denote a field of characteristic zero with a fixed prime ideal
above $2$.
Then  exactly the jacobian elliptic  fibrations \#6 and \#8 reduce smoothly
to $X$ up to isomorphism over $\bar k$.
\end{Proposition}

\begin{proof}
Let $S\to \PP^1$ be an elliptic surface over $k$.
In order for this specific elliptic fibration to have good reduction,
the singular fibers are only allowed to  degenerate from multiplicative
type  to additive type, but never with additional fiber components
(only irreducible fibers (nodal and cuspidal) and types $\ta_1,
\ta_2$).

In the present situation, $X$ is supersingular with $\rho(X)=22$,
but in characteristic  zero $\rho(S)\leq h^{1,1}(S)=20$.
Hence in case of good reduction, the Picard number can only be
increased by additional sections.
In general this gives
\[
\mbox{rank}(\MW(X\to \PP^1))\geq \rho(X)-\rho(S)\geq 2.
\]
But in the present situation, \#6 and \#8 are the only elliptic
fibrations on $X$ with $\MW$ rank at least two.
In fact, we have equality, so any elliptic lift $S$ must have $\rho(S)=20$
and finite $\MW$ (i.e.~it is extremal).
In particular, this implies that the configurations of reducible
singular fibers coincide in characteristic zero and $2$.
(In characteristic zero, \#6 also has three singular fibers of
type $I_1$;
upon reduction mod $2$, these singular fibers are indeed merged
with the $\td_7$ fiber,
but the degeneration only contributes to the wild ramification
\cite{S-MJM}.)
Over an algebraically closed field, each configuration determines
a unique elliptic surface, and the equations from \#6, \#8 do
in fact work in any characteristic other than $3$.
\end{proof}

\begin{Remark}
Over non-algebraically closed fields (such as number fields, finite
fields),
there are cubic twists occurring.
See~\cite{S-MJM} for an analysis  over $\Q$
 that  generalizes directly to other fields.
\end{Remark}

\begin{Remark}
A singular K3 surface with supersingular  good reduction
automatically leads to Artin invariant one by \cite[Proposition 1.0.1]{Shimada}.
Thus we infer from Proposition \ref{prop:0}
that \#6 and \#8 give the only jacobian elliptic singular K3 surfaces
with supersingular good reduction at a prime above $2$.
\end{Remark}

\subsection*{Acknowledgements}

We thank the referee for her or his comments.
During the preparation of this manuscript,
both authors enjoyed the hospitality of each other's home institution
whom we thank for the support.
This work was started when the second author held a position at
Copenhagen University.

\end{document}